\documentclass[12pt,a4paper,reqno]{amsart}
\usepackage{amsfonts,amsthm,amsmath,amscd}
\usepackage{booktabs}
\usepackage[utf8]{inputenc}
\usepackage{t1enc}
\usepackage[mathscr]{eucal}
\usepackage{indentfirst}
\usepackage{tikz}
\usepackage{algorithm}
\usepackage{algpseudocode}
\usetikzlibrary{shapes.geometric, positioning, calc}
\usetikzlibrary{automata,arrows,positioning,calc}
\usetikzlibrary{positioning}
\usepackage{caption}
\usepackage{subcaption}
\usepackage[font=small,labelfont=bf]{caption}
\usepackage{graphicx,graphics,pict2e}
\usepackage{tkz-berge}
\usetikzlibrary{arrows.meta} 
\usetikzlibrary{decorations.pathmorphing,decorations.markings} 
\usepackage{float}
\usepackage{epic}
\usepackage{lipsum}
\usepackage{stmaryrd}
\usepackage{authblk}
\usepackage[symbol]{footmisc}
\usepackage[normalem]{ulem}
\numberwithin{equation}{section}
\usepackage[margin=2.9cm]{geometry}
\usepackage{epstopdf}
\RequirePackage{doi}
\usepackage{hyperref}
\allowdisplaybreaks
\usepackage{cite}
\usepackage{multirow}
\usepackage{multicol}
\usepackage{rotating}
\usepackage{verbatim,amssymb,amsfonts}
\usepackage{mathrsfs}
\usepackage{mathtools}

\usepackage{tikz-cd}
\usepackage{indentfirst}
\usepackage{slashed}
\usepackage{thmtools}
\usepackage{setspace}
\usepackage{enumerate}
\usepackage{accents}
\usepackage{pdflscape}

\usetikzlibrary{arrows.meta, bending}

\usepackage{appendix}
\theoremstyle{plain}
\newtheorem{theorem}{Theorem}[section]
\newtheorem{lemma}[theorem]{Lemma}
\newtheorem{corollary}[theorem]{Corollary}
\newtheorem{proposition}[theorem]{Proposition}

\newtheorem{notation}{Notation}

 \theoremstyle{definition}
\newtheorem{definition}[theorem]{Definition}
\newtheorem{remark}[theorem]{Remark}

\DeclarePairedDelimiterX{\inp}[2]{\langle}{\rangle}{#1, #2}

\renewcommand{\>}{\rangle}

\newcommand{\cA}{{\mathcal A}}

\newcommand{\cE}{{\mathcal E}}

\newcommand{\cV}{{\mathcal V}}

\newcommand{\ba}{\begin{eqnarray}}
\newcommand{\na}{\end{eqnarray}}
\newcommand{\ban}{\begin{eqnarray*}}
\newcommand{\nan}{\end{eqnarray*}}

\newcommand{\R}{{\mathbb R}}
\newcommand{\Z}{{\mathbb Z}}

\renewcommand{\a}{\alpha}

\renewcommand{\thefootnote}{\fnsymbol{footnote}}

\makeatletter
\g@addto@macro{\endabstract}{\@setabstract}
\newcommand{\authorfootnotes}{\renewcommand\thefootnote{\@fnsymbol\c@footnote}}%
\makeatother

\makeatletter
\@namedef{subjclassname@2020}{%
  \textup{2020} Mathematics Subject Classification}
\makeatother

\title[]{Lin--Lu--Yau Ricci Curvature of Digraphs \\ via Optimal Transport Couplings}



\begin{document}

\begin{center}
    \vspace{-1cm}
	\maketitle
	
	\normalsize
    \authorfootnotes
    Kevin Fung, Johnny Lim\footnote[1]{Corresponding author.}
	\par \bigskip

        \small{School of Mathematical Sciences, Universiti Sains Malaysia, Penang, Malaysia}\par \bigskip
\end{center}

\address{School of Mathematical Sciences, Universiti Sains Malaysia, Penang, Malaysia
}
\email{fungkevin@student.usm.my}
\email{johnny.lim@usm.my}

\vspace{-0.5cm}
\begin{abstract}
In this paper, we study the Lin--Lu--Yau Ricci curvature of strongly connected locally finite digraphs through an explicit optimal-coupling construction. For an arc of a digraph, we derive a computable curvature formula by constructing a coupling between the probability measures at its tail and head, and by proving its optimality using a suitable $1$-Lipschitz function. The formula is not only effective for direct computation, but also unifies several known results: in particular, it recovers the Lin--Lu--Yau Ricci curvature formula for Cayley graphs of Right-Angled Artin--Coxeter Hybrid groups as a special case and gives shorter proofs of curvature results arising from matching-type conditions. We then characterize arcs with zero Ricci curvature through perfect distance matching and perfect distance partitions. 
We further prove that, under suitable assumptions, such arc curvature in directed Cayley graphs increases when an inverse generator or a new generator is added to the generating set. As applications, we compute the  curvature of directed Cayley graphs of dihedral groups and generalized quaternion groups, including $\Gamma(D_n,\{a,b\})$, $\Gamma(Q_{4m},\{a,b\})$, $\Gamma(Q_{4m},\{a,a^{-1},b\})$ and $\Gamma(Q_{4m},\{a,b,b^{-1}\})$. Finally, we provide an algorithm for computing the Lin--Lu--Yau Ricci curvature of Cayley graphs of finitely generated groups with prescribed generating sets, together with complete curvature tables for several important families of finite groups.

\end{abstract}

\section{Introduction}
\label{sec1}

Discrete curvature has become an important tool for studying graphs through geometric ideas. It provides a bridge between classical curvature in Riemannian geometry and discrete structures such as finite graphs, Markov chains, and Cayley graphs. Several notions of curvature on graphs have been developed, including the Bakry–Émery curvature \cite{bakryemery1985diffusion}, the Ollivier Ricci curvature \cite{ollivier2009Riccimetric}, the Lin–Lu–Yau Ricci curvature \cite{linluyau2011Ricci}, the Lott–Sturm–Villani curvature \cite{sturm2006geometryI,sturm2006geometryII,lott2009Riccimetric}, and the Steinerberger curvature \cite{steinerberger2023curvature}. These notions are based on different viewpoints, such as curvature-dimension inequalities, optimal transport, metric-measure theory, equilibrium measures and distance matrices. Among them, Ollivier Ricci curvature and Lin–Lu–Yau Ricci curvature are especially useful for measuring how local neighborhoods of adjacent vertices are transported toward each other.

Let $G=(\cV,\cE)$ be an undirected graph equipped with a probability measure $\mu$. For two vertices $x,y\in \cV$, Ollivier \cite{ollivier2009Riccimetric} defined the coarse Ricci curvature  $\kappa(x,y)$ by the relation
\begin{equation}
\label{eq_ollivier_ricci}
    \frac{W(\mu_x, \mu_y)}{d(x,y)}\coloneqq 1-\kappa (x,y),
\end{equation}
where $\mu_x = \dfrac{\mu|_{B(x,1)}}{\mu(B(x,1))}$ is the restriction of $\mu$ to the $1$-ball of $x,$ 
and $W(\mu_x, \mu_y)$ is the $1$-Wasserstein distance from $\mu_x$ to $\mu_y$. Later, Lin, Lu and Yau \cite{linluyau2011Ricci} introduced the $\alpha$-Ricci curvature $k_\a (x,y)$ using Equation \eqref{eq_ollivier_ricci} by choosing $\mu_x$ to be $\alpha$-lazy random walks. 
The Lin-Lu-Yau Ricci curvature is then defined by the limiting formula 
\begin{equation}
     \kappa(x,y) = \lim_{\alpha \rightarrow 1 } \frac{\kappa_{\alpha} (x,y)}{1-\alpha}.
\end{equation}
This definition has been widely studied for undirected graphs because it is closely related to transport, local graph structure, and several comparison-type results.

The directed setting is more delicate. In an undirected graph, both directions between two adjacent vertices are present, while in a digraph the geometry of an arc depends on the orientation and on the out-neighborhood structure.
In 2019, Yamada \cite{yamada2019Riccidirected} extended the Lin-Lu-Yau Ricci curvature to digraphs. However, explicit computation remains difficult, since it requires identifying the exact $1$-Wasserstein distance between probability measures associated with two adjacent vertices.

The main technical contribution of this paper is an explicit formula for the Lin–Lu–Yau Ricci curvature of arcs in strongly connected digraphs, cf. Theorem \ref{theo_curvature_regular}. 
The proof is based on an explicit coupling $A$ between the probability measures associated with the initial and terminal vertices of an arc. To prove that this coupling is optimal, we construct a suitable $1$-Lipschitz function $f$ adapted to $A$, so that the upper and lower bounds for the $1$-Wasserstein distance coincide.
This turns out to be extremely useful as it gives a concrete method for computing Ricci curvature from the local distance structure around an arc.
More specifically, Theorem \ref{theo_curvature_regular} recovers:
\begin{enumerate}[(1).]
    \item Hehl's formula of the Lin-Lu-Yau Ricci curvature for undirected graphs (Theorem \ref{hehl2026regularpositivericci})
    \item the Lin-Lu-Yau Ricci curvature for Cayley graphs of Right Angled Artin-Coxeter Hybrids (RAACHs) group (Proposition \ref{lemma_cushing2025})
    \item non-positive Ricci curvature of strongly connected digraphs with disjoint out-neighbourhoods (Corollary \ref{corollary_nonpositivecurvature})
    \item the Lin-Lu-Yau Ricci curvature of undirected graphs satisfying Local Matching and Extended Matching Conditions (Corollary \ref{cor3.10}-\ref{cor3.12}).
\end{enumerate}

In fact, (4) is a special case under the newly introduced notions called \textit{perfect distance matching} and \textit{perfect distance partitions}. We also utilize the former notion to prove a sufficient condition for when the Ricci curvature of an arc of a directed Cayley graph increases, cf. Theorems \ref{theo_increasecurv_inverse} and \ref{theo_curvature_add_gen}.
Moreover, using Theorem \ref{theo_curvature_regular}, we establish a characterization of the arcs with vanishing Ricci curvature in Theorem \ref{theo_arc_zerocurvature}. In particular, Ricci curvature vanishes exactly when the out-neighborhoods of $x$ and $y$ can be paired in such a way that makes the optimal transport cost attain the flat-curvature value.




To demonstrate the practicality of Theorem \ref{theo_curvature_regular}, we apply it to directed Cayley graphs arising from several important families of finite groups. For the dihedral group $D_n$ with $n\geq 3$, we compute the Ricci curvature of $\Gamma(D_n,\{a,b\})$ in Proposition \ref{prop_ricci_dihedral}. We also consider generalized quaternion groups and obtain explicit Ricci curvature formulas for $\Gamma(Q_{4m},\{a,b\})$, $\Gamma(Q_{4m},\{a,a^{-1},b\})$ and $\Gamma(Q_{4m},\{a,b, b^{-1}\})$, as stated in Propositions \ref{prop_ricci_quartenion_{a,b}}, 
\ref{prop_ricci_quartenion_3_a^{-1}}, and 
\ref{prop_ricci_quartenion_3_b^{-1}}, respectively. Finally, we present an algorithm for computing the curvature of Cayley graphs of the $\mathbb Z_n$, the symmetric group $S_n$ and the alternating group $A_n$. Remarkably, our algorithm is able to reproduce the results done by I. Mizukai and A. Sako\cite{mizukaiRicci_Cayley2024} with greater efficiency. 

This paper is organized as follows. Sect. \ref{sect_preliminaries} recalls the preliminaries on directed graphs, directed Cayley graphs, couplings and Lin--Lu--Yau Ricci curvature. Sect. \ref{sect_curv} proves Theorem \ref{theo_curvature_regular} for strongly connected locally finite digraphs, introduces perfect distance matching and partitions, and gives applications of the theorem. In Subsect. \ref{subsect_newresults}, we study how the Ricci curvature of arcs in directed Cayley graphs changes when inverse generators or new generators are added. Sect. \ref{sect_computation} computes the curvature for Cayley graphs of dihedral and generalized quaternion groups. A complete curvature tables for Cayley graphs of $D_n$, $Q_{4m}$, $\mathbb Z_n$, $S_n$ and $A_n$ are included in the Appendix.

\section{Preliminaries}
\label{sect_preliminaries}
In this section, we list necessary and relevant preliminaries to be used throughout. 

\begin{definition}(\cite{yamada2019Riccidirected})
Let $D=(\cV,\cA)$ be a directed graph. 
    \begin{enumerate}
        \item For any two vertices $x,y \in \cV$, if there is a path from $x$ to $y$, then $y$ is said to be reachable from $x$. If this holds for all vertices $x,y\in \cV$, then $D$ is said to be \textit{strongly connected}. The \textit{distance} $d(x,y)$ is the length of shortest directed path from $x$ to $y$. If there is no such path, then we define $d(x,y)=\infty$.

        \item For any vertex $x\in \cV$, the \textit{out-degree} $d_x^{\text{out}}$ of $x$ is the number of arcs initiated from $x$. $D$ is \textit{locally finite} if every vertex has a finite out-degree.

        \item The set $N^{\text{out}}(x)=\{v \in \cV : (x,v) \in \cA \}$ is called the \textit{out-neighborhood} of $x$. In particular, $d_x^{\text{out}}=|N^{\text{out}}(x)|$.

        \item $D$ is \textit{$r$-regular} if all vertices have the same out-degree $r$.

        \item $D$ is \textit{simple} if it has no multiple arcs and loops.
    \end{enumerate}
\end{definition}

\vspace{-0.3cm}
\begin{notation}
Henceforth, for simplicity, the  term ``\textbf{\textit{digraph}}'' shall mean strongly connected locally finite simple directed graphs, unless stated otherwise.
\end{notation}

\vspace{-0.3cm}
\begin{definition}
(\cite{grossman1964groups})
\label{def_directed Cayley graphs}
Let $G$ be a group and let $S \subseteq G-\{e\}$ be a generating set of $G$. A \textit{directed Cayley graph} $\Gamma (G, S)$ is defined as a simple digraph with vertex set $G$ and arcs of the form $(g,gs)$ for every $g \in G$ and $s \in S$.
\end{definition}

\vspace{-0.3cm}
\begin{definition}
(\cite{grossman1964groups})
\label{def_undirected Cayley graphs}
Let $G$ be a group and let $S \subseteq G-\{e\}$ be a generating set of $G$ such that $S$ is symmetric (inverse-closed), i.e. $S=S^{-1}.$ The Cayley graph $\Gamma(G, S)$ is then regarded as an \textit{undirected} simple graph by identifying symmetric arcs between two vertices as one edge.
\end{definition}

\vspace{-0.3cm}
\begin{remark}
In general, the undirected (or directed) Cayley graphs $\Gamma(G,S)$ are regular of degree $|S|$. In Definition \ref{def_directed Cayley graphs} and \ref{def_undirected Cayley graphs}, if $S$ generates $G$, i.e. $G=\<S\>$, then $G$ is connected. In general, if $S$ does not generate $G$, then $G$ is disconnected. If the unit $e$ is allowed in $S$, then every $g \in G$ has a self-loop since $(g,g)=(g,ge)$. 
\end{remark}

\begin{definition} {(\cite{bjorner2006combinatorics}, \cite{Johnson1980groupspresentation})}
    \begin{enumerate}
        \item For $n\geq 3$, the dihedral groups $D_n$ are defined as 
        \begin{equation*}
            D_n=\langle \ a,b \mid a^n=b^2=e \text{ and } ba=a^{n-1}b \ \rangle.
        \end{equation*}

        \item For $m\geq 2$, the generalized quaternion groups $Q_{4m}$ are defined as 
        \begin{equation*}
            Q_{4m}= \left< a,b \mid a^{2m}=e, b^2=a^m \text{ and } b^{-1}ab=a^{-1} \right>.
        \end{equation*}

        \item For $n\geq 2$, the integer modulo $n,$ $\Z_n,$ is defined as 
        \begin{equation*}
            \Z_n = \langle\ 1 \mid 1(n)=e \ \rangle.
        \end{equation*}

        \item For $n\geq 2$, the symmetric group $S_n$ on $n$ objects is defined as
        \begin{equation*}
            S_n = \langle \ \sigma_1, \ldots, \sigma_{n-1} \mid \sigma_i^2=e, \sigma_i \sigma_j = \sigma_j \sigma_i, \text{ for }  |i-j|>1, (\sigma_i \sigma_{i+1})^2=e \ \rangle.
        \end{equation*}

        \item For $n\geq 3$, the alternating group $A_n$ is defined as
        \begin{equation*}
            A_n = \langle \ V_1, \ldots, V_{n-2} \mid V_i^3=e, (V_i V_{i+1})^2=e, \text{ for }  1\leq i<j\leq n-2 \ \rangle.
        \end{equation*}
    \end{enumerate}
\end{definition}


\begin{definition}{(\cite{yamada2019Riccidirected})}
\label{def_alphalazywalk}
   Let $D=(\cV,\cA)$ be a digraph 
   and let $\alpha \in [0,1].$ The \textit{$\alpha$-lazy random walk} probability measure  $\mu^{\alpha}_x : \cV \to [0,1]$ is defined as 
    \begin{align*}
    \mu^{\alpha}_x(v)=
        \begin{cases}
	\alpha, &\text{if } v=x,\\
	\dfrac{1-\alpha}{d^{\text{out}}_x}, &\text{if } (x,v) \in \cA, \\
    0, &\text{otherwise}.
    \end{cases}
    \end{align*}
\end{definition}

\begin{definition}{(\cite{yamada2019Riccidirected})}
    \label{def_Wasserstein}
    For two probability measures $\mu$ and $\nu$ on $\cV$, the $1$-Wasserstein distance between $\mu$ and $\nu$ is defined as
    \begin{equation}
          \label{eq_Wasserstein_distance}
          W(\mu, \nu) = \inf_{A} \sum_{x,y \in \cV} A(x,y) d(x,y),
    \end{equation}
    where $A:\cV \times \cV \rightarrow [0,1]$, called a \textit{coupling} between $\mu$ and $\nu,$ is a map satisfying
    \begin{equation}
    \label{eq_coupling_criteria}
    \begin{cases}
    \sum_{y\in \cV} A(x,y) \ = \ \mu(x),\\
	\sum_{x\in \cV} A(x,y) \ = \ \nu(y).\\
    \end{cases}
    \end{equation}
\end{definition}

\begin{remark}
\begin{enumerate}[(i).]
    \item Since the directed distance is not necessarily symmetric, the resulting $1$-Wasserstein distance may also be non-symmetric.
    \item A coupling $A$ that attains the $1$-Wasserstein distance is called an \textit{optimal} coupling. Since $\mu_x^\alpha$ and $\mu_y^\alpha$ have finite supports, the coupling set $\Pi(\mu_x^\alpha,\mu_y^\alpha)$ defined by \eqref{eq_coupling_criteria} is a nonempty compact polytope, and the transport cost $A\mapsto \sum_{u,v}A(u,v)d(u,v)$ is a continuous linear functional. Hence an optimal coupling exists by the Weierstrass Extreme Value Theorem. Moreover, by the linear programming result of Bazaraa--Jarvis--Sherali \cite[\S 3]{bazaraa2010linear}, one may choose an optimal coupling at an extreme point of the coupling polytope. However, an optimal coupling need not be unique.

    \item     For any coupling $B$, we have $\sum_{v\in \cV} B(v,w)=\mu_{y}^\alpha (w)$. For $w \in \cV -( N^{\text{out}} (y) \cup{\{y\}} )$, we have $\mu_{y}^\alpha (w)=0$ by definition. It follows that $\sum_{v\in \cV} B(v,w)=0$. As $B$ takes value in $[0,1]$, we have $B(v,w)=0$ for every $v \in \cV$. Therefore, the condition
    $\sum_{w\in \cV} B(v,w)=\mu_{x}^\alpha (v)$
reduces to 
\begin{align*}
    \sum_{w\in N^{\text{out}} (y) \cup{\{y\}} } B(v,w)=\mu_{x}^\alpha (v)
\end{align*}
for every $v\in \cV$. Similarly, for $v\in \cV -( N^{\text{out}} (x) \cup{\{x\}} )$, we have $B(v,w)=0$ for every $w\in V$. 
\end{enumerate}
\end{remark}

\begin{definition}{(\cite{yamada2019Riccidirected})}
    For any two distinct vertices $x,y\in \cV$, the $\alpha$-Ricci curvature of $x$ and $y$, $\kappa_{\alpha}(x,y)$ is defined as 
    \begin{equation*}
        \kappa_{\alpha} (x,y) = 1- \frac{W(\mu_x^{\alpha}, \mu_y^{\alpha})}{d(x,y)}.
    \end{equation*}
    In the case where $y$ is not reachable from $x$, we define $k_\alpha(x,y) := -\infty$.
    The Lin-Lu-Yau Ricci curvature of $x$ and $y$ is defined as 
    \begin{equation*}
        \kappa(x,y) = \lim_{\alpha \rightarrow 1 } \frac{\kappa_{\alpha} (x,y)}{1-\alpha}.
    \end{equation*}
\end{definition}

    By choosing a coupling, Definition \ref{def_Wasserstein} provides an upper bound of $W(\mu_x^{\alpha}, \mu_y^{\alpha})$. On the other hand, a lower bound of $W(\mu_x^{\alpha}, \mu_y^{\alpha})$ can be obtained below.

\begin{proposition}{(\cite{yamada2019Riccidirected})}
\label{prop_W_sup}
    For any two vertices $x$ and $y$, we have
    \begin{equation*}
        W(\mu_x^{\alpha}, \mu_y^{\alpha}) \geq \sup_{f} \left( \sum_{v\in V} f(v)\mu_x^{\alpha}(v) - \sum_{w\in V} f(w)\mu_x^{\alpha}(w)  \right), 
    \end{equation*}
    where $f:\cV \rightarrow \R$ is function such that $f(v)-f(w) \leq d(v,w)$.
\end{proposition}

\begin{remark}
    For a regular digraph with regularity $r$, the expression 
    \begin{equation*}
        \sum_{v\in V} f(v)\mu_x^{\alpha}(v) - \sum_{w\in V} f(w)\mu_x^{\alpha}(w), 
    \end{equation*}
    can be rewritten as 
    \begin{equation}
        [f(x)-f(y)]\alpha+\frac{1-\alpha}{r} \left[ \sum_{v\in N^{\text{out}}(x)} f(v) - \sum_{w\in N^{\text{out}}(y)} f(w)     \right].  
    \end{equation}
\end{remark}


\section{Optimal couplings for Lin-Lu-Yau Ricci curvature of digraphs} \label{sect_curv}
In this section, we give an explicit construction of an optimal coupling for the Lin–Lu–Yau Ricci curvature of digraphs. The main idea is to reduce the computation of the $1$-Wasserstein distance to a suitable bijection between the out-neighborhoods $N^{out}(x)$ and $N^{out}(y)$, under the assumption that the two vertices have the same out-degree. This bijection is chosen so that the total directed distance between paired vertices is minimized, and it leads to the curvature formula below.

\begin{theorem}
\label{theo_curvature_regular}
    Let $D=(\cV, \cA)$ be a  digraph. Let $(x,y)$ be an arc of $D$ such that $d^{\text{out}}_x=d^{\text{out}}_y=r$. Then, the Lin-Lu-Yau Ricci curvature $\kappa(x,y)$ is given by
    \begin{equation}
     \label{eq_lin-lu-yau_curvature}
    \kappa(x,y)=
    \begin{cases}
        1-\dfrac{1}{r} \left(  \displaystyle \sum_{v\in N^{\text{out}}(x)} d(v,w(v))   \right), &\text{ if $x\notin N^{\text{out}}(y)$} \\[20pt]
        
        1-\dfrac{1}{r} \left(  \displaystyle \sum_{v\in N^{\text{out}}(x)-\{y\} } d(v,w(v)) -1  \right), &\text{ if $x\in N^{\text{out}}(y).$}
    \end{cases}
    \end{equation}
    If $x\notin N^{\text{out}}(y)$, then for each $v\in N^{\text{out}}(x)$, there exists a unique $w(v)\in N^{\text{out}}(y)$ such that there is an one-to-one correspondence between $N^{\text{out}}(x)$ and $N^{\text{out}}(y)$ with $\sum_{v\in N^{\text{out}}(x)} d(v,w(v))$ being minimum among all such possible choices. If $x\in N^{\text{out}}(y)$, then $w(y)$ is always chosen to be $x$.
\end{theorem}

\begin{proof}
We prove by cases.
\begin{enumerate}[(a)]

        \item If $x\notin N^{\text{out}} (y)$ (i.e., arc $(y,x)$ does not exist), then we define $A(x,y)=\alpha$. Suppose $N^{\text{out}} (x) = \{v_1, \ldots, v_r\}$, where $y=v_i$ for some $i,$ and $N^{\text{out}} (y) =\{w_1, \ldots, w_r\}$. For each $v_i$, we choose a unique vertex $w(v_i) \in N^{\text{out}} (y) $ such that for $i \neq j$, $w(v_i) \neq w(v_j)$ with $\sum_{i=1}^{r}d(v_i, w(v_i))$ being the minimum among all such choices. This is justified as both $N^\text{out}(x)$ and $N^\text{out}(y)$ are finite and of the same size. In other words, we have a bijection from $N^{\text{out}} (x)$ to $N^{\text{out}} (y)$. If a vertex $u$ is in both $N^{\text{out}} (x) $ and $N^{\text{out}} (y),$ i.e., $x$ and $y$ have a common out-neighbor, then we choose $w(u)$ to be $u$.

        Define $A(v_i, w(v_i)):=\frac{1-\alpha}{r}$ for all $i$. 
        For all other ordered pairs $(s,t) \in \cV \times \cV$, $A(s,t)$ is defined to be zero.
        Then, we have
\begin{align*}    
    \sum_{v,w\in \cV} A(v,w)d(v,w) &= A(x,y)d(x,y)+\sum_{i=1}^r \sum_{j=1}^r A(v_i,w_j)d(v_i,w_j) \\
    &= \alpha+\sum_{i=1}^r A(v_i,w(v_i))d(v_i,w(v_i)) \\
    &= \alpha+ \sum_{i=1}^r \left( \frac{1-\alpha}{r} \right) d(v_i, w(v_i)) \\
    &= \alpha+ \frac{1-\alpha}{r}  \sum_{v\in N^{\text{out}} (x) } d(v, w(v)).
\end{align*}

        \item If $x\in N^{\text{out}} (y)$ (i.e., arc $(y,x)$ exists), then we define 
        \[
        A(x,y)= \alpha-\frac{1-\alpha}{r}, \ A(x,x)=\frac{1-\alpha}{r} = A(y,y), \ A(y,x)=0.
        \] Note that for $y$, we choose $w(y)$ to be $x$. For value of coupling $A$ on $ \left( \cV-\{x,y\} \right) \times \left( \cV-\{x,y\} \right)$, refer to $(a)$. 
        Now, let $y=v_1$, then we have 
\begin{align*}    
    \sum_{v,w\in \cV} A(v,w)d(v,w) &= A(x,y)d(x,y)+\sum_{i=1}^r \sum_{j=1}^r A(v_i,w_j)d(v_i,w_j) \\
    &= \alpha-\frac{1-\alpha}{r}+A(y,x)d(y,x)+\sum_{i=2}^r A(v_i,w(v_i))d(v_i,w(v_i)) \\
    &= \alpha-\frac{1-\alpha}{r}+ \sum_{i=2}^r \left( \frac{1-\alpha}{r} \right) d(v_i, w(v_i)) \\
    &= \alpha+ \frac{1-\alpha}{r} \left( \sum_{v\in N^{\text{out}} (x)-\{y\} } d(v, w(v)) -1 \right).
\end{align*}
\end{enumerate}

Now, we define a function $f: \cV \rightarrow \R$ such that the expression
        \begin{equation*}
            \sum_{v\in V} f(v)\mu_x^{\alpha}(v) - \sum_{w\in V} f(w)\mu_x^{\alpha}(w) 
        \end{equation*}
        in Proposition \ref{prop_W_sup}
        equals to 
        \begin{equation*}
             \sum_{x,y \in \cV} A(x,y) d(x,y).
        \end{equation*}
Recall that $f$ needs to satisfy $f(a)-f(b) \leq d(a,b)$. Suppose we have a bijection from $N^{\text{out}} (x)$ to $N^{\text{out}} (y)$ mapping $v$ to $w(v)$. Define $f:V \rightarrow \R$ as follows:
\begin{enumerate}[(i)]
    \item Case (a): $x\notin N^{\text{out}} (y)$, then
    \begin{enumerate}[1.]
        \item Let $f(x)=1, f(y)=0$ and $f(w(y))=-1$.
        
        \item Suppose $N^{\text{out}} (x)-\{y\}=\{v_2, \ldots, v_r\}$ and $N^{\text{out}} (y)-\{w(y)\}=\{w(v_2), \ldots, w(v_r)\}$. Let $f(v_2)=\min_{ v\in \{x,y,w(y)\}  } \{ f(v)+d(v_2,v) \}$ and $f(w(v_2))=f(v_2)-d(v_2,w(v_2))$.

        \item For $2\leq k < r$, let $S_k=\{x,y,w(y), v_2,\ldots,v_k, w(v_2),\ldots,w(v_k)\}$. Suppose we have defined $f$ on $S_k$ such that it satisfies $f(a)-f(b) \leq d(a,b)$ for every $a,b \in S_k$. Define 
        \[
        f(v_{k+1})= \min_{v\in S_k} \{ f(v)+d(v_{k+1},v) \}, 
        \]
        \[
        f(w(v_{k+1}))=f(v_{k+1})-d(v_{k+1}, w(v_{k+1})).
        \]
        Inductively, it follows that we have defined a function $f$ on $N^{\text{out}} (x) \cup N^{\text{out}} (y)\cup \{x,y\}$ such that $f(a)-f(b) \leq d(a,b)$ holds. Then this function can be extended to $V$ easily.
    \end{enumerate}
Observe that 
\begin{align*}
    [f(x)-f(y)]&\alpha + \frac{1-\alpha}{r}  \sum_{v\in N^{\text{out}}(x)} [ f(v) - f(w(v)) ] \\
    &=\alpha + \frac{1-\alpha}{r}\sum_{v\in N^{\text{out}}(x)} d(v,w(v)) 
    = \sum_{v,w\in \cV} A(v,w)d(v,w).
\end{align*}

    \item Case (b): $x\in N^{\text{out}} (y)$. Steps $2$ and $3$ remain the same as in case (a) except for Step 1. Since $w(y)=x$, we simply define $f(x)=1$ and $f(y)=0$. Then

\begin{align*}
    &[f(x)-f(y)]\alpha+\frac{1-\alpha}{r}  \sum_{v\in N^{\text{out}}(x)} [ f(v) - f(w(v)) ] \\
    &= [f(x)-f(y)]\alpha+\frac{1-\alpha}{r} \left( f(y)-f(x) + \sum_{v\in N^{\text{out}}(x)-\{y\} }[ f(v) - f(w(v)) ] \right) \\
    &=\alpha + \frac{1-\alpha}{r} \left( 0-1  + \sum_{v\in N^{\text{out}}(x)-\{y\} }  d(v,w(v))   \right) \\
    &= \alpha+ \frac{1-\alpha}{r} \left( \sum_{v\in N^{\text{out}} (x) - \{y\} } d(v, w(v)) -1 \right) 
    = \sum_{v,w\in \cV} A(v,w)d(v,w).
\end{align*}
\end{enumerate}
From above, we have deduced that the coupling $A$ is indeed an optimal coupling. It follows that $W(\mu_x^{\alpha}, \mu_y^{\alpha} ) = \sum_{x,y \in \cV} A(x,y) d(x,y)$. 

For $x\notin N^{\text{out}}(y)$. Since the coupling $A$ is optimal, it follows that 
    \[
    W(\mu_x^\a, \mu_y^\a ) 
    = \sum_{v,w\in \cV} A(v,w)d(v,w) \\ 
    = \alpha+ \frac{1-\alpha}{r} \left( \sum_{v\in N^{\text{out}} (x) } d(v, w(v)) \right).
    \]
    Therefore, 
    \[
    \dfrac{1-W(\mu_x^\a, \mu_y^\a )}{(1-\a)}= 1-\dfrac{1}{r} \left( \sum_{v\in N^{\text{out}}(x)} d(v,w(v))\right)
    \] and hence
    \[
    \kappa(x,y)
    = \lim_{\a \rightarrow 1} \dfrac{1-W(\mu_x^\a, \mu_y^\a )}{(1-\a)} 
    = 1-\dfrac{1}{r} \left( \sum_{v\in N^{\text{out}}(x)} d(v,w(v))   \right). 
    \]
    The case of  $x\in N^{\text{out}}(y)$ is similar. The proof is now complete.
\end{proof}

\begin{remark}
    If $d^{\text{out}}_x \neq d^{\text{out}}_y$, then the above construction cannot be applied since there does not exist a bijection between $N^{\text{out}} (x) $ and $N^{\text{out}} (y)$. 

    For a fixed vertex $v\in N^{\text{out}} (x)$, instead of choosing multiple vertices $w_{j_1}, \ldots, w_{j_{m}}$ such that $A(v,w_{1})+\cdots + A(v,w_{j_m})=\frac{1-\alpha}{r}$, which might increase the transport cost $\sum_{k=1}^{m} A(v, w_{j_k})d(v,w_{j_k})$. It is ideal to just choose one $w$ while taking into consideration that the sum $\sum_{i=1}^r d(v_i, w_{l(i)})$ needs to be minimum. This is exactly the reason the coupling is defined based on the bijection between $N^{\text{out}} (x) $ and $N^{\text{out}} (y)$.
\end{remark}

From the construction of coupling above, we define the following new notion.
\begin{definition}
\label{def_perfectdistancematching}
    Let $D=(\cV,\cA)$ be an undirected or directed simple graph and $U$ and $W$ be two nonempty subsets of $V$ (not necessary disjoint) with $|U|=|W|$. A \textit{perfect distance matching} from $U$ to $W$ is a bijection $F:U \rightarrow W$ such that $\sum_{u\in U} d(u,F(u))$ is minimum among all bijections from $U$ to $W$.
\end{definition}

Note that if $F$ is a perfect distance matching from $U$ to $W$, then its inverse $F^{-1}$ is not necessary a perfect distance matching from $W$ to $U$. A perfect distance matching $F$, if exists, can be viewed as a set of ordered pairs 
\[
M(F)=\{(u, F(u)) : u\in U\}.
\]
If $D$ is undirected and $d(u,F(u))=1$ holds for every $u\in U,$ which forces $U$ and $W$ to be disjoint, then $M(F)$ corresponds to a matching of $U\cup W$. If we further assume that $U \cup W= \cV$ so that $|U|=|W|=\frac{|\cV|}{2}$, then we have a perfect matching on $D$. Theorem \ref{theo_curvature_regular} essentially requires us to find a perfect distance matching from $N^{\text{out}}(x)$ (resp. $N^{\text{out}}(x)-\{y\}$) to $N^{\text{out}}(y)$ (resp. $N^{\text{out}}(y)-\{x\}$). 

If $u \in U \cap W$, then a perfect distance matching $F$ necessarily maps $u$ to itself.

\begin{lemma}
Let $F:U \rightarrow W$ be a perfect distance matching from $U$ to $W$, if $u \in U \cap W$, then $F$ maps $u$ to itself, i.e. $F(u)=u$.
\end{lemma}

\begin{proof}
    Suppose on the contrary that $F(u)\neq u$. Let $v\in U$ such that $F(v)=u$. Define $G:U \rightarrow W$ such that $G(u)=u, G(v)=F(u),$ and $G(x)=F(x)$ for $x\in U-\{u,v\}$. Then,
    \begin{align*}
        \sum_{x\in U} d(x,G(x)) &=  \sum_{x\in U-\{u,v\}}  d(x,G(x)) + d(u,G(u))+d(v,G(v)) \\
         &=  \sum_{x\in U-\{u,v\}}  d(x,F(x)) + d(u,u)+d(v,F(u)) \\
         &\leq \sum_{x\in U-\{u,v\}}  d(x,F(x)) + 0+d(v,u)+d(u,F(u)) \\
         &= \sum_{x\in U} d(x,F(x)).
    \end{align*}
    This contradicts the assumption that $F$ is a perfect distance matching.  
\end{proof}
The above lemma justifies the choice in Theorem \ref{theo_curvature_regular} that if $x$ and $y$ has a common out-neighbor, say $u$, we always choose $w(u)$ to be $u$.

Now, we introduce a generalization of Extended Matching Condition defined in \cite{dagli2019extendedmathing}.

\begin{definition}
\label{def_perfectdistancepartitions}
    Let $D=(\cV,\cA)$ be a digraph.
    Let $(x,y) \in \cA$ be an arc such that $x \notin N^{\text{out}}(y)$.  A \textit{perfect distance partition} of $N^{\text{out}}(x)$ and $N^{\text{out}}(y),$ respectively, is the sets $\{P_1,\ldots, P_k\}$ and $\{Q_1,\ldots, Q_k\}$ with
    \begin{align}
    N^{\text{out}}(x) - N^{\text{out}}(y) &= P_1 \cup P_2 \cup \cdots P_k, \label{eq:perfdistpart1}\\
    N^{\text{out}}(y) - N^{\text{out}}(x)  &= Q_1 \cup Q_2 \cup \cdots Q_k, \label{eq:perfdistpart2}
    \end{align}
    and such that 
    \begin{enumerate}[(i)]
        \item for each $i=1,\ldots,k$, there exists a bijection $f_i: P_i \rightarrow Q_i$ with $d(v,f_i(v))=d_i$ for every $v\in P_i$; 

        \item for every $v\in P_i$ and $w\in Q_j$, $d(v,w)\geq d(v,f(v_i))=d_i$.
    \end{enumerate}
    For the case of $x \in N^{\text{out}}(y)$, the definition is modified by replacing $N^{\text{out}}(y)$ with $N^{\text{out}}(y) \cup \{y\}$ in \eqref{eq:perfdistpart1} and $N^{\text{out}}(x)$ with $N^{\text{out}}(x) \cup\{x\}$ in \eqref{eq:perfdistpart2}. 
\end{definition}

\begin{proposition}
Let $D=(\cV,\cA)$ be a digraph. Suppose that $(x,y) \in \cA$ is an arc that admits perfect distance partitions for $N^{\text{out}}(x)$ and $N^{\text{out}}(y)$. Then, the Lin-Lu-Yau Ricci curvature of $(x,y)$ is 
\begin{equation}
    \kappa(x,y)
    =\begin{cases}
        1-\dfrac{1}{|N^{\text{out}}(x)|}\displaystyle \sum_{i=1}^k |P_i|d_i,  
        &\quad \text{if $x \notin N^{\text{out}}(y),$ }\\[20pt]

        1-\dfrac{1}{|N^{\text{out}}(x)|}\left ( \displaystyle\sum_{i=1}^k |P_i|d_i- 1 \right), &\quad \text{if $x \in N^{\text{out}}(y).$ }
    \end{cases}
\end{equation}
\end{proposition}

\begin{proof}
    It is easy to check that if we have a perfect distance partition, then $d^{\text{out}}(x)=d^{\text{out}}(y)$. For $x \notin N^{\text{out}}(y)$, we define a bijection $F$ from $N^{\text{out}}(x)$ to $N^{\text{out}}(y)$ by 
    \begin{enumerate}[(i)]
    \item $F|_{P_i}=f_i$,
    \item $F|_{N^{\text{out}}(x) \cap N^{\text{out}}(y)}= \mathrm{Id}$.
    \end{enumerate}

    It follows from Definition \ref{def_perfectdistancepartitions} that $F$ is a perfect distance matching from $N^{\text{out}}(x)$ to $N^{\text{out}}(y)$ with 
    \[
    \sum_{v \in N^{\text{out}}(x)} d(v,F(v)) =\sum_{i=1}^k |P_i|d_i.
    \]
    By Theorem \ref{theo_curvature_regular}, we have
    
    \begin{equation}
    \label{eq_k(x,y)_pdp_(a)}
    \kappa (x,y)=1-\frac{1}{|N^{\text{out}}(x)|} \sum_{i=1}^k |P_i|d_i .
    \end{equation}

    For $x \in N^{\text{out}}(y)$, by a similar argument, we get
    \begin{equation}
    \label{eq_k(x,y)_pdp_(b)}
    \kappa (x,y)=1-\frac{1}{|N^{\text{out}}(x)|}\left (\sum_{i=1}^k |P_i|d_i -1\right).
    \end{equation}
    This completes
    the proof.
\end{proof}

\subsection{Application of Theorem \ref{theo_curvature_regular}} \label{subsect_appl}


In this subsection, we shall demonstrate how Theorem \ref{theo_curvature_regular} can be applied to recover some recent results in \cite{hehl2026regularpositivericci, yamada2019Riccidirected, smith2014matching, dagli2019extendedmathing, cushing2025BEandRicci}.

For undirected locally finite simple graphs, Hehl \cite{hehl2026regularpositivericci} obtained an explicit formula for the Lin-Lu-Yau Ricci curvature as follows:
\begin{theorem}\cite[Theorem 1.1]{hehl2026regularpositivericci}
\label{hehl2026regularpositivericci}
Let $G=(V,E)$ be a locally finite graph. Let $x,y \in V$ be of equal degree $r$ with $x \sim y$. Then, the Lin–Lu–Yau Ricci curvature $\kappa(x,y)$ is 
    \begin{equation}
    \kappa(x,y)=\dfrac{1}{r} \left(r+1-\inf_{\phi \in \cA_{xy}} \sum_{v\in S_1(x)-B_1(y)} d(v,\phi(v)) \right),
    \end{equation}
where $\cA_{xy}$ denotes the set of all bijections $\phi : S_1(x)-B_1(y) \rightarrow S_1(y)-B_1(x)$.
\end{theorem}
We show that Theorem \ref{theo_curvature_regular} reduces to Theorem \ref{hehl2026regularpositivericci} when $G$ is undirected.
\begin{proof}
Since $G$ is undirected, we have $x\in N^{\text{out}}(y)=N(y)$. Our formula gives
\begin{align*}
      \kappa(x,y) &= 1-\dfrac{1}{r} \left(  \displaystyle \sum_{v\in N(x)-\{y\} } d(v,w(v)) -1  \right) \\
      &= \dfrac{1}{r} \left( r+1 - \sum_{v\in N(x)-\{y\} } d(v,w(v))  \right).
\end{align*}
Recall that if a vertex $v$ is in both $N(x)$ and $N(y)$, then $w(v)$ is always chosen to be $v$, giving $d(v,w(v))=0$. Thus, $\kappa(x,y)$ reduced to 
\begin{equation*}
      \kappa(x,y) = \dfrac{1}{r} \left( r+1 - \sum_{v\in S_1(x)-B_1(y) } d(v,w(v))  \right).
\end{equation*}
Notice that $N(x)=S_1(x)$ and $B_1(y)=N(y) \cup \{y\}$. 

If $\phi:S_1(x)-B_1(y) \rightarrow S_1(y)-B_1(x)$ is any bijection, then our choice guarantees that 
$\sum_{v\in S_1(x)-B_1(y) } d(v,w(v)) \leq \sum_{v\in S_1(x)-B_1(y) } d(v,\phi(v))$. In other words, 
\begin{align*}
    \sum_{v\in S_1(x)-B_1(y) } d(v,w(v))=\inf_{\phi} \sum_{v\in S_1(x)-B_1(y) } d(v,\phi(v)),
\end{align*}
where infimum is taken over all bijections $\phi:S_1(x)-B_1(y) \rightarrow S_1(y)-B_1(x)$. Therefore,
\begin{equation}
    \kappa(x,y)=\dfrac{1}{r} \left(r+1-\inf_{\phi} \sum_{v\in S_1(x)-B_1(y)} d(v,\phi(v)) \right),
\end{equation}
with infimum taken over all bijections $\phi:S_1(x)-B_1(y) \rightarrow S_1(y)-B_1(x)$. This completes the proof.
\end{proof}



\begin{corollary}{\cite[Corollary 2.13]{yamada2019Riccidirected}}
\label{corollary_nonpositivecurvature}
    Assume the hypotheses of Theorem \ref{theo_curvature_regular}. Suppose further that $N^{\text{out}}(x) \cap N^{\text{out}}(y) = \emptyset$ and $x \notin N^{\text{out}}(y)$, then the Lin-Lu-Yau Ricci curvature of arc $(x,y)$ is always non-positive, i.e. $\kappa(x,y)\leq 0$.
\end{corollary}

\begin{proof}
    Suppose that $\kappa(x,y)>0$. Since $x \notin N^{\text{out}}(y)$, by Theorem \ref{theo_curvature_regular}, we have 
    \begin{align*}
          \kappa(x,y)= 1-\dfrac{1}{r} \left( \sum_{v\in N^{\text{out}}(x)} d(v,w(v))    \right) >0.  
    \end{align*}
    This implies that 
    \[
    \sum_{v\in N^{\text{out}}(x)} d(v,w(v)) < r.
    \]
    However,
    $\sum_{v\in N^{\text{out}}(x)} d(v,w(v))\geq 1\cdot |N^{\text{out}}(x)|=r$ as $N^{\text{out}}(x) \cap N^{\text{out}}(y) = \emptyset$, which is a contradiction. Hence, $\kappa(x,y) \leq 0$.
\end{proof}

\begin{definition}
Suppose $G=(\cV, \cE)$ is a simple undirected graph. Let $\{x,y\}$ be an edge of $G$ with $d_x=d_y=r$.
    \begin{enumerate}
    \item In \cite{smith2014matching}, the edge $\{x,y\}$ is said to satisfy the \textit{Local Matching condition} if there is a perfect matching between $N(x) - (\{y\} \cup N(y))$ and $N(y) - (\{x\} \cup N(x))$ within the subgraph induced by $N(x) \cup N(y) \cup \{x,y\}$. \\    
    This is equivalent to the existence of a perfect distance partitions $\{P_1\}$ for $N^{\text{out}}(x) - (\{y\} \cup N^{\text{out}}(y))$ and $\{Q_1\}$ for $N^{\text{out}}(y) - (\{x\} \cup N^{\text{out}}(x))$  with $d_1=1$.

    \item In \cite{dagli2019extendedmathing}, the \textit{Extended Matching Condition of Type $2$} is equivalent to 
        \begin{enumerate}
        \item $N^{\text{out}}(x) - (\{y\} \cup N^{\text{out}}(y)) = P_1\cup \{p\}$.

        \item $N^{\text{out}}(y) - (\{x\} \cup N^{\text{out}}(x)) = Q_1\cup \{q\}$.
        \end{enumerate}
        with $d_1=1$ and $d(p,q)=3$. 

    \item In \cite{dagli2019extendedmathing}, the \textit{Extended Matching Condition of Type $3$}  is equivalent to \begin{enumerate}
    \item $N^{\text{out}}(x) - (\{y\} \cup N^{\text{out}}(y)) = P_1\cup \{p\}$.

    \item $N^{\text{out}}(y) - (\{x\} \cup N^{\text{out}}(x)) = Q_1\cup \{q\}$.
\end{enumerate}
with $d_1=1$ and $d(p,q)=2$.
    \end{enumerate}
\end{definition}

\begin{corollary}{\cite[Theorem 6.3]{smith2014matching}}
\label{cor3.10}
Suppose an edge $\{x,y\}$ with $d_x=d_y=r$ satisfies the Local Matching Condition. Let $r$ be the common degree of $x$ and $y.$ Then,
\begin{equation}
    \kappa(x,y)= \frac{2+|N^{\text{out}} (x) \cap N^{\text{out}} (y)|}{r}.
\end{equation}
\end{corollary}

\begin{proof}
     By Equation \eqref{eq_k(x,y)_pdp_(b)}, we obtain 
    \begin{align*}
        \kappa (x,y)&=1-\frac{1}{|N^{\text{out}}(x)|}\left (\sum_{i=1}^k |P_i|d_i -1\right) \\
        &=1-\frac{1}{r}(r-1- |N^{\text{out}} (x) \cap N^{\text{out}} (y)|-1) \\
        &= \frac{2+|N^{\text{out}} (x) \cap N^{\text{out}} (y)|}{r}. \qedhere
    \end{align*}
\end{proof}

\begin{corollary}{\cite[Theorem 2.7]{dagli2019extendedmathing}}
\label{cor3.11}
Suppose an edge $\{x,y\}$ with $d_x=d_y=r$ satisfies the Extended Matching Condition of Type $2$. Then,
\[
\kappa (x,y) = \frac{1+|N^{\text{out}} (x) \cap N^{\text{out}} (y)|}{r}.
\]
\end{corollary}

\begin{proof}
    By Equation \eqref{eq_k(x,y)_pdp_(b)}, we obtain
    \begin{align*}
        \kappa (x,y)&=1-\frac{1}{|N^{\text{out}}(x)|}\left (\sum_{i=1}^k |P_i|d_i -1\right) \\
        &=1-\frac{1}{r}(r- |N^{\text{out}} (x) \cap N^{\text{out}} (y)|-1) \\
    &= \frac{1+|N^{\text{out}} (x) \cap N^{\text{out}} (y)|}{r}. \qedhere
    \end{align*}
\end{proof}

\begin{corollary}{\cite[Theorem 2.10]{dagli2019extendedmathing}}
\label{cor3.12}
Suppose an edge $\{x,y\}$ with $d_x=d_y=r$ satisfies the Extended Matching Condition of Type $3$. Then,
\[
\kappa (x,y) = \frac{|N^{\text{out}} (x) \cap N^{\text{out}} (y)|}{r}.
\]
\end{corollary}

\begin{proof}
    By Equation \eqref{eq_k(x,y)_pdp_(b)}, we obtain
    \begin{align*}
        \kappa (x,y)&=1-\frac{1}{|N^{\text{out}}(x)|}\left (\sum_{i=1}^k |P_i|d_i -1\right) \\
        &=1-\frac{1}{r}(r- |N^{\text{out}} (x) \cap N^{\text{out}} (y)|+1-1) \\
    &= \frac{|N^{\text{out}} (x) \cap N^{\text{out}} (y)|}{r}. \qedhere
    \end{align*}

\end{proof}

Next, the lemma below is crucial in determining the Lin-Lu-Yau Ricci curvature for Cayley graphs of \textit{Right Angled Artin-Coxeter Hybrids} (RAACHs) groups. We give a shorter proof by using Theorem \ref{theo_curvature_regular}. 

\begin{proposition}{\cite[Lemma 20]{cushing2025BEandRicci}}
\label{lemma_cushing2025}
Let $G = (V,E)$ be an undirected graph with edge $\{x,y\}$. Assume that the neighborhood structures of x and y are as illustrated in Figure \ref{fig:cushinglemma20}, with
or without an additional common vertex z (illustrated in blue), so the degrees of x
and y are equal and either $n+l+2$ or $n+l+1$ (depending on whether $z \in V$ or $z \notin V$ ).
Suppose that
\begin{enumerate}[i.]
    \item $d(x_i, y_j) = 3$ for $i,j \in [l]$,
    \item $d(x_i, v_j) = d(y_i, u_j) = 3$ for $i\in [l]$ and $j\in [n]$,
    \item $d(u_j,v_j) = 1 $ for $ j\in [n]$.
\end{enumerate}
Then, we have
\begin{align*}
    \kappa(x,y)
    =\begin{cases}
        \dfrac{3-2l}{l+n+2}, &\text{ if $z \in V$,} \\[20pt]
        \dfrac{2-2l}{l+n+1}, &\text{ if $z \notin V.$}
    \end{cases}
\end{align*}
\end{proposition}

\begin{proof}
    Suppose $z\in V$. From Figure \ref{fig:cushinglemma20}, it is clear that the pairs $(x_i,y_i), (u_j,v_j)$ for $i\in [l]$ and $j\in [n],$ and $(z,z)$ minimize 
    \[
    \sum_{v\in N^{\text{out}}(x)}  d(v,w(v))=3l+n.
    \] By Theorem \ref{theo_curvature_regular}, we have
        \[
        \kappa(x,y) 
        = 1-\dfrac{1}{l+n+2}(3l+n-1) 
        =\frac{3-2l}{l+n+2}.
        \]
    The case of $z\notin V$ can be argued in a similar fashion.
\end{proof}

\begin{figure}[htp!]
\begin{center}
\begin{tikzpicture}[
    scale=1.0,
    vertex/.style={circle, fill=black, inner sep=1.2pt},
    bluevertex/.style={circle, fill=blue!80!black, inner sep=1.5pt},
    group/.style={ellipse, draw, minimum width=1.8cm, minimum height=3.7cm, thick}
]

    \node[vertex, label=below left:{$x$}] (x) at (0,0) {};
    \node[vertex, label=below right:{$y$}] (y) at (4,0) {};
    \draw[thick] (x) -- (y);

    \node[bluevertex, label=above:{\color{blue!80!black}$z$}] (z) at (2,2) {};
    \draw[dashed, blue!80!black, thick] (x) -- (z);
    \draw[dashed, blue!80!black, thick] (y) -- (z);

    \node[group] (left_group) at (-1.5,0.8) {};
    \node[vertex, label=left:{$x_1$}] (x1) at (-1.5,0) {};
    \node[vertex, label=left:{$x_2$}] (x2) at (-1.5,0.6) {};
    \node at (-1.5,1.2) {$\vdots$};
    \node[vertex, label=above:{$x_\ell$}] (xl) at (-1.5,1.8) {};
    
    \foreach \n in {x1, x2, xl} \draw (x) -- (\n);

    \node[group] (right_group) at (5.5,0.8) {};
    \node[vertex, label=right:{$y_1$}] (y1) at (5.5,0) {};
    \node[vertex, label=right:{$y_2$}] (y2) at (5.5,0.6) {};
    \node at (5.5,1.2) {$\vdots$};
    \node[vertex, label=above:{$y_\ell$}] (yl) at (5.5,1.8) {};
    
    \foreach \n in {y1, y2, yl} \draw (y) -- (\n);

    \node[group] (u_group) at (1,-2.5) {};
    \node[group] (v_group) at (3,-2.5) {};

    \node[vertex, label=above right:{$u_1$}] (u1) at (0.9,-1.6) {};
    \node[vertex, label=above right:{$u_2$}] (u2) at (1,-2.2) {};
    \node at (1,-2.7) {$\vdots$};
    \node[vertex, label=above right:{$u_n$}] (un) at (1,-3.4) {};

    \node[vertex, label=above left:{$v_1$}] (v1) at (3.1,-1.6) {};
    \node[vertex, label=above left:{$v_2$}] (v2) at (3,-2.2) {};
    \node at (3,-2.7) {$\vdots$};
    \node[vertex, label=above left:{$v_n$}] (vn) at (3,-3.4) {};

    \foreach \n in {u1, u2, un} \draw (x) -- (\n);
    \foreach \n in {v1, v2, vn} \draw (y) -- (\n);
    
    \draw (u1) -- (v1);
    \draw (u2) -- (v2);
    \draw (un) -- (vn);

\end{tikzpicture}
\end{center}
    \caption{Neighborhood structures of $x$ and $y$ in Lemma \ref{lemma_cushing2025}.}
    \label{fig:cushinglemma20}
\end{figure}
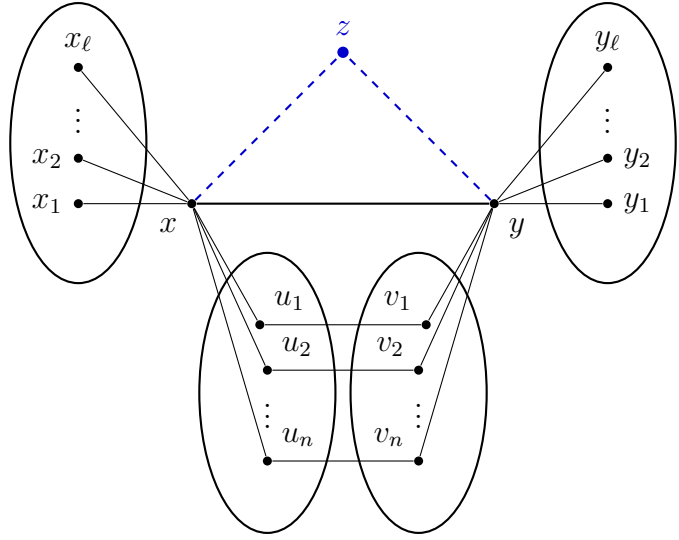

\subsection{New results on the Lin-Lu-Yau Ricci curvature of arcs of digraphs} \label{subsect_newresults}

Using Theorem \ref{theo_curvature_regular}, we now characterize arcs of strongly connected digraphs with zero Ricci curvature.

\begin{theorem}
\label{theo_arc_zerocurvature}
    Assuming the hypotheses of Theorem \ref{theo_curvature_regular}. Suppose further that $N^{\text{out}}(x) \cap N^{\text{out}}(y) = \emptyset$. Then, the following statements are equivalent
    \begin{enumerate}
        \item $\kappa(x,y)=0$.

        \item 
        \begin{enumerate}
            \item For case $x\notin N^{\text{out}}(y)$, there exists a bijection from $N^{\text{out}}(x)$ to $N^{\text{out}}(y)$ such that $d(v,w(v))=1$ holds for every $v\in N^{\text{out}}(x)$ with $d^{\text{out}}_x \geq 1$.

        \item For case $x\in N^{\text{out}}(y)$, there exists a bijection from $N^{\text{out}}(x)$ to $N^{\text{out}}(y)$ mapping $y$ to $x$ such that there exists either $u_0 \in N^{\text{out}}(x) -\{y\} $ with $d(u_0, w(u_0))=3$ and $d(v,w(v))=1$ for any other $v\in N^{\text{out}}(x)-\{u_0,y\}$ with $d^{\text{out}}_x \geq 2$, or $u_0,v_0 \in N^{\text{out}}(x) -\{y\}$ with $d(u_0, w(u_0))=d(v_0, w(v_0))=2$ and $d(v,w(v))=1$ for any other $v\in N^{\text{out}}(x)-\{u_0,v_0,y\}$ with $d^{\text{out}}_x \geq 3$.
        \end{enumerate}
    \end{enumerate}
\end{theorem}

\begin{proof}
    The implication $(2)(a)\rightarrow(1)$ is proved in  \cite[Theorem 4.4]{yamada2019Riccidirected}. $(2)(b)\rightarrow(1)$ can be checked by direct computation. We show that $(1)$ implies $(2)$.

    Suppose $k(x,y)=0$. For case $x\notin N^{\text{out}}(y)$, By Theorem \ref{theo_curvature_regular}, it follows that
    \[
    \sum_{v\in N^{\text{out}}(x)} d(v,w(v)) = r.
    \]
    Since $N^{\text{out}}(x) \cap N^{\text{out}}(y) = \emptyset$ and $d(v,w(v))\geq1$ for every $v\in N^{\text{out}}(x)$, this forces $d(v,w(v))=1$, which is $(2)(a)$ (See Figure \ref{fig:k=0, case (a)}).

    For the case $x\in N^{\text{out}}(y)$, again by Theorem \ref{theo_curvature_regular}, we have
    \[
    \sum_{v\in N^{\text{out}}(x)-\{y\} } d(v,w(v)) = r+1.
    \]
    Observe that $|N^{\text{out}}(x)-\{y\}|=r-1$ and $ d(v,w(v)) \geq 1$. Since $(r-1)+2=r+1$, either there exists a vertex $u_0$ such that $d(u_0,w(u_0))=3$ and $d(v,w(v))=1$ for any other $v\in N^{\text{out}}(x)-\{u_0,y\}$ with $d^{\text{out}}_x \geq 2$, or there exists $u_0,v_0 \in N^{\text{out}}(x) -\{y\}$ such that $d(u_0, w(u_0))=d(v_0, w(v_0))=2$ and $d(v,w(v))=1$ for any other $v\in N^{\text{out}}(x)-\{u_0,v_0,y\}$ with $d^{\text{out}}_x \geq 3$ (See Figure \ref{fig:k=0, case (b)}). This completes the proof.

\end{proof}

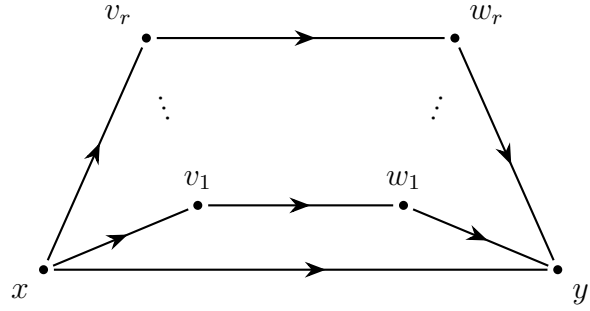
\begin{figure}[htp!]
\begin{center}
\begin{tikzpicture}[scale=1.7, 
    vertex/.style={circle, fill=black, inner sep=1.2pt},
    directed/.style={
        black, thick, 
        shorten >=2pt, shorten <=2pt,
        postaction={decorate, decoration={markings, mark=at position 0.55 with {\arrow{Stealth[scale=1.2]}}}}
    }]

    \node[vertex, label=below left:{$x$}] (x) at (0,0) {};
    \node[vertex, label=below right:{$y$}] (y) at (4,0) {};

    \draw[directed] (x) -- (y);

    \node[vertex, label=above:{$v_1$}] (v1) at (1.2,0.5) {};
    \node[vertex, label=above:{$w_1$}] (w1) at (2.8,0.5) {};
    \draw[directed] (x) -- (v1);
    \draw[directed] (v1) -- (w1);
    \draw[directed] (w1) -- (y);

    \node[black, label={[rotate=20]above:$\vdots$}] at (1.0,1.0) {};
    \node[black, label={[rotate=-20]above:$\vdots$}] at (3.0,1.0) {};
    
    \node[vertex, label=above left:{$v_r$}] (vr) at (0.8,1.8) {};
    \node[vertex, label=above right:{$w_r$}] (wr) at (3.2,1.8) {};
    \draw[directed] (x) -- (vr);
    \draw[directed] (vr) -- (wr);
    \draw[directed] (wr) -- (y);

\end{tikzpicture}
\end{center}
    \caption{Case (a): $x \notin N^{out}(y)$ with $r\geq 1$.}
    \label{fig:k=0, case (a)}
\end{figure}


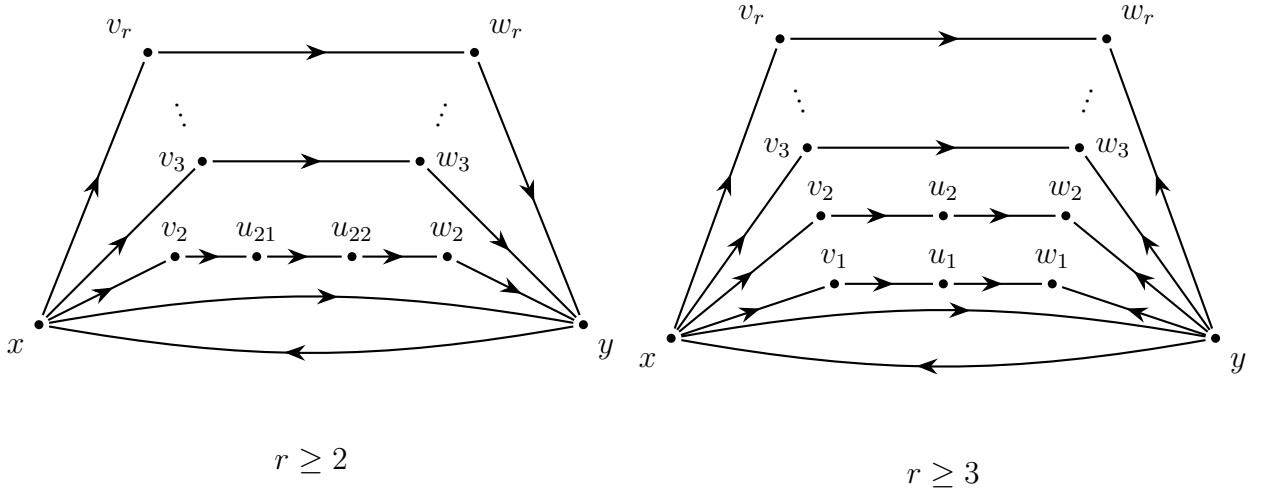
\begin{figure}[htbp]
    \centering
    \begin{minipage}{0.45\textwidth}
        \centering
    \begin{tikzpicture}[scale=1.8, 
    vertex/.style={circle, fill=black, inner sep=1.2pt},
    directed/.style={
        black, thick, 
        shorten >=2pt, shorten <=2pt,
        postaction={decorate, decoration={markings, mark=at position 0.55 with {\arrow{Stealth[scale=1.2]}}}}
    }]

    \node[vertex, label=below left:{$x$}] (x) at (0,0) {};
    \node[vertex, label=below right:{$y$}] (y) at (4,0) {};

    \draw[directed] (x) to[bend left=10] (y);
    \draw[directed] (y) to[bend left=10] (x);

    \node[vertex, label=above :{$v_2$}] (v2) at (1,0.5) {};
    \node[vertex, label=above:{$u_{21}$}] (u21) at (1.6,0.5) {};
    \node[vertex, label=above:{$u_{22}$}] (u22) at (2.3,0.5) {};
    \node[vertex, label=above :{$w_2$}] (w2) at (3,0.5) {};
    
    \draw[directed] (x) -- (v2);
    \draw[directed] (v2) -- (u21);
    \draw[directed] (u21) -- (u22);
    \draw[directed] (u22) -- (w2);
    \draw[directed] (w2) -- (y);

    \node[vertex, label= left:{$v_3$}] (v3) at (1.2,1.2) {};
    \node[vertex, label= right:{$w_3$}] (w3) at (2.8,1.2) {};
    \draw[directed] (x) -- (v3);
    \draw[directed] (v3) -- (w3);
    \draw[directed] (w3) -- (y);

    \node[black, label={[rotate=20]above:$\vdots$}] at (1.1,1.3) {};
    \node[black, label={[rotate=-20]above:$\vdots$}] at (2.9,1.3) {};

    \node[vertex, label=above left:{$v_r$}] (vr) at (0.8,2.0) {};
    \node[vertex, label=above right:{$w_r$}] (wr) at (3.2,2.0) {};
    \draw[directed] (x) -- (vr);
    \draw[directed] (vr) -- (wr);
    \draw[directed] (wr) -- (y);

    \node at (2, -1.0) {$r \geq 2$};
\end{tikzpicture}
    \end{minipage}
    \hfill 
    \begin{minipage}{0.45\textwidth}
        \centering
   \begin{tikzpicture}[scale=1.8, 
    vertex/.style={circle, fill=black, inner sep=1.2pt},
    directed/.style={
        black, thick, 
        shorten >=2pt, shorten <=2pt,
        postaction={decorate, decoration={markings, mark=at position 0.55 with {\arrow{Stealth[scale=1.2]}}}}
    }]

    \node[vertex, label=below left:{$x$}] (x) at (0,0) {};
    \node[vertex, label=below right:{$y$}] (y) at (4,0) {};

    \draw[directed] (x) to[bend left=10] (y);
    \draw[directed] (y) to[bend left=10] (x);

    \node[vertex, label=above:{$v_1$}] (v1) at (1.2,0.4) {};
    \node[vertex, label=above:{$u_1$}] (u1) at (2,0.4) {};
    \node[vertex, label=above:{$w_1$}] (w1) at (2.8,0.4) {};
    \draw[directed] (x) -- (v1);
    \draw[directed] (v1) -- (u1);
    \draw[directed] (u1) -- (w1);
    \draw[directed] (y) -- (w1);

    \node[vertex, label=above :{$v_2$}] (v2) at (1.1,0.9) {};
    \node[vertex, label=above:{$u_2$}] (u2) at (2,0.9) {};
    \node[vertex, label=above:{$w_2$}] (w2) at (2.9,0.9) {};
    \draw[directed] (x) -- (v2);
    \draw[directed] (v2) -- (u2);
    \draw[directed] (u2) -- (w2);
    \draw[directed] (y) -- (w2);

    \node[vertex, label=left:{$v_3$}] (v3) at (1.0,1.4) {};
    \node[vertex, label=right:{$w_3$}] (w3) at (3.0,1.4) {};
    \draw[directed] (x) -- (v3);
    \draw[directed] (v3) -- (w3);
    \draw[directed] (y) -- (w3);

    \node[black, label={[rotate=20]above:$\vdots$}] at (1.0,1.5) {};
    \node[black, label={[rotate=-20]above:$\vdots$}] at (3.0,1.5) {};

    \node[vertex, label=above left:{$v_r$}] (vr) at (0.8,2.2) {};
    \node[vertex, label=above right:{$w_r$}] (wr) at (3.2,2.2) {};
    \draw[directed] (x) -- (vr);
    \draw[directed] (vr) -- (wr);
    \draw[directed] (y) -- (wr);

    \node at (2, -1.0) {$r \geq 3$};
\end{tikzpicture}
    \end{minipage}
    \caption{Case (b):  $x \in N^{out}(y)$.}
    \label{fig:k=0, case (b)}
\end{figure}


By utilizing Definition \ref{def_perfectdistancematching} and Theorem \ref{theo_curvature_regular},  we establish a sufficient condition for a \textit{monotonicity} result of directed Cayley graphs: for when its Lin-Lu-Yau Ricci curvature increases.

\begin{theorem}
\label{theo_increasecurv_inverse}
    Let $G$ be a finitely generated group with $S=\{s_1,s_2, \ldots, s_r\}$ being a set of generators satisfying certain relation $R$. Suppose $s_k \in S$ satisfies $s_k^{-1} \notin S$ and cannot be expressed as $s_i s_j^{-1}$ for every $s_i,s_j \in S$. If $k(e,s_k)$ (resp. $k'(e,s_k)$) is the Lin-Lu-Yau Ricci curvature of arc $(e,s_k)$  in $\Gamma (G,S)$ (resp. $\Gamma (G,S')$ with $S'=S\cup \{s_k^{-1} \}$), then $k'(e,s_k) > k(e,s_k).$
\end{theorem}

\begin{proof}
    Without loss of generality, we assume $s_k=s_1$. Consider $\kappa'(e,s)$ in $\Gamma (G,S')$ first. Note that $N^{\text{out}}_{S'} (e)=\{s_1, s_1^{-1}, s_2, \ldots s_r\}$ and $N^{\text{out}}_{S'} (s_1)=\{e, s_1^{2}, s_2, \ldots s_r\}$. Let $f$ be a perfect distance matching from $N^{\text{out}}_{S'} (e)-\{s_1\}$ to $N^{\text{out}}_{S'} (s_1)-\{e\}$.

    Meanwhile, consider $\kappa(e,s)$ in $\Gamma(G,S)$. Let $g$ be a perfect distance matching from $N^{\text{out}}_S (e)=\{s_1, s_2, \ldots,s_r \}$ to $N^{\text{out}}_S(s_1)=\{s_1^2, s_1s_2, \ldots,s_1s_r\}$. Define a bijection  $g' : N^{\text{out}}_{S'} (e)-\{s_1\} \rightarrow N^{\text{out}}_{S'} (s_1)-\{e\}$ by $g'(v)=g(v)$ for $v \neq s_1^{-1}$ and $g'(s_1^{-1})=g(a)$. Since $f$ is a perfect distance matching, we must have 
    \begin{align*}
        \sum_{v \in N^{\text{out}}_{S'} (e) -\{s_1\} } d_{S'}(v,f(v)) &\leq \sum_{v \in N^{\text{out}}_{S'} (e) -\{s_1\} } d_{S'}(v,g'(v)) \\
        &= \sum_{v \in N^{\text{out}}_{S'} (e) -\{s_1, s^{-1}\} } d_{S'}(v,g'(v)) + d_{S'}(s_1^{-1}, g'(s_1^{-1}) ) \\
        &\leq \sum_{v \in N^{\text{out}}_{S} (e) } d_{S}(v,g(v)) -d_{S'}(s_1, g(s_1)) + d_{S'}(s_1^{-1}, g'(s_1^{-1}) ) \\
        &\leq \sum_{v \in N^{\text{out}}_{S} (e) } d_{S}(v,g(v)) + d_{S'}(s_1^{-1}, s_1) \\
        &= \sum_{v \in N^{\text{out}}_{S} (e) } d_{S}(v,g(v)) + 2.
    \end{align*}

Since $e \in  N^{\text{out}}_{S'} (s_1)$ and $e \notin  N^{\text{out}}_{S} (s_1)$, by Theorem \ref{theo_curvature_regular}
\begin{equation*}
    \kappa' (e,s_1) = 1-\frac{1}{r+1} \left(\sum_{v \in N^{\text{out}}_{S'} (e) -\{s_1\} } d_{S'}(v,f(v)) -1 \right),
\end{equation*}
and 
\begin{equation*}
    \kappa (e,s_1) = 1-\frac{1}{r} \left(\sum_{v \in N^{\text{out}}_{S} (e) } d_{S}(v,g(v)) \right).
\end{equation*}

It follows that 
\[
(r+1)(1-\kappa'(e,s_1)) \leq r(1-\kappa (e,s_1))+1 ,
\]
or equivalently,
\[
\kappa'(e,s_1) > \frac{r}{r+1} \kappa (e,s_1).
\]
The assumption that $s_1$ is not expressible in the form of $s_is_j^{-1}$ implies that $N^{\text{out}}_S (e) \cap N^{\text{out}}_S(s_1)=\emptyset$. By Corollary \ref{corollary_nonpositivecurvature}, we have $\kappa (e,s_1)\leq 0$. Hence, $\kappa' (e,s_1)> \kappa (e,s_1) - \frac{1}{r+1}\kappa (e,s_1) \geq  \kappa (e,s_1)$. The proof is now complete.
\end{proof}

\begin{remark}
    Theorem \ref{theo_increasecurv_inverse} can be used to explain the pattern in Propositions \ref{prop_ricci_quartenion_{a,b}}, \ref{prop_ricci_quartenion_3_a^{-1}}, \ref{prop_ricci_quartenion_3_b^{-1}} in Sect. \ref{sect_computation}. Observe that $\kappa(e,a)$ increases from $\Gamma(Q_{4m}, \{a,b\})$ to $\Gamma(Q_{4m}, \{a,a^{-1},b\})$. Similarly, $\kappa(e,b)$ increases from $\Gamma(Q_{4m}, \{a,b\})$ to $\Gamma(Q_{4m}, \{a,b,b^{-1}\})$. 
\end{remark}

By a similar argument in Theorem \ref{theo_increasecurv_inverse}, one obtains: 
\begin{theorem}
\label{theo_curvature_add_gen}
    Let $G$ be a finitely generated group with $S=\{s_1,s_2, \ldots, s_r\}$ being a set of generators satisfying certain relation $R$. Suppose $s_{r+1} \notin S$ is another generator of $G$. For any $i=1,\ldots,r$, let $k(e,s_i)$ (resp. $k'(e,s_i)$) be the Lin-Lu-Yau Ricci curvature of arc $(e,s_i)$ in $\Gamma (G,S)$ (resp.  $\Gamma (G,S')$ with $S'=S\cup \{s_{r+1} \}$). 
    Suppose 
    \[
    \kappa (e,s_i) \leq 1-d_{S'}(s_{r+1},s_is_{r+1}),
    \]
    then $k'(e,s_i) \geq k(e,s_i).$ 
\end{theorem}

\begin{lemma}
\label{lemma_curvature_single_generator}
    Given a finitely generated group $G$ and a generating set $S=\{a\}$, where $a$ is one of the generators. Then, 
    \begin{align*}
    \kappa(e,a)
    =\begin{cases}
        0 & \text{ if $|a|\neq 2$ }\\
        2 & \text{ if $|a|= 2$. }
    \end{cases}
    \end{align*}
\end{lemma}

\begin{proof}
    Suppose first that $|a|\neq 2$. Note that $N^{\text{out}}(e)=\{a\}$ and $N^{\text{out}}(a)=\{a^2\}$. Since $d(a,a^2)=1$, we have $\kappa(e,a)=1-\frac{1}{1}(1)=0$. 

    Now if $|a|=2$, then $N^{\text{out}}(e)=\{a\}$ and $N^{\text{out}}(a)=\{a^2=e\}$. It follows that $\kappa(e,a)=1-\frac{1}{1}(0-1)=2$.
\end{proof}

\section{Computation of Ricci Curvature of Cayley Graphs} \label{sect_computation}
In this section, we compute the Lin--Lu--Yau Ricci curvature of several directed Cayley graphs. We first consider the directed Cayley graph $\Gamma(D_n,\{a,b\})$ of the dihedral group $D_n$ for $n\geq 3$, where we explicitly demonstrate the construction given in Theorem \ref{theo_curvature_regular}. We then compute the Ricci curvature of $\Gamma(Q_{4m},\{a,b\})$, $\Gamma(Q_{4m},\{a,a^{-1},b\})$ and $\Gamma(Q_{4m},\{a,b,b^{-1}\})$ for the generalized quaternion group $Q_{4m}$, where $m\geq 2$.

For the generalized quaternion cases, we apply Theorem \ref{theo_curvature_regular} by first determining the directed distances between vertices in the out-neighborhood of the initial vertex and vertices in the out-neighborhood of the terminal vertex. These distances are then used to choose the pairings which minimize the total transport cost. The computations also illustrate the effect of enlarging the generating set. In particular, adding the inverse generator $a^{-1}$ or $b^{-1}$ increases the Lin-Lu-Yau Ricci curvature of the corresponding arc $(e,a)$ or $(e,b)$, in accordance with Theorem \ref{theo_increasecurv_inverse}.

Finally, we present an algorithm for computing the Lin-Lu-Yau Ricci curvature of Cayley graphs of finitely generated groups with prescribed generating sets. Several resulting curvature tables are included in the Appendix. 

Throughout this section, we write $W$ for $W(\mu_x^{\alpha},\mu_y^{\alpha})$ when the measures are clear from the context.

\begin{proposition}
\label{prop_ricci_dihedral}
    The Lin-Lu-Yau Ricci curvature for the directed Cayley graphs $\Gamma(D_n, S')$ with generating set $S'=\{a,b\}$ is 
\begin{center}
\begin{tabular}{ | c | c | c | } 
  \hline
  & \multicolumn{2}{c|}{$\Gamma(D_n, \{a,b\})$} \\
  \hline
  $\kappa(x,y)$  & $n=3$          & $n\geq 4$  \\ 
  \hline
  $\kappa(e,a) $ & $-\frac{1}{2}$ & $-1$ \\ 
  \hline
  $\kappa(e,b)$  & $\frac{1}{2}$  & $0$ \\ 
  \hline
\end{tabular}
\end{center}
\end{proposition}

\begin{proof}
   We compute $\kappa (e,a)$ first. For $n\geq 3$, we have $N^{\text{out}}(e)=\{a,b\}$ and $N^{\text{out}}(a)=\{a^2,ba^{n-1}\}$. Since $e\notin N^{\text{out}}(a)$, we apply case (a) in the proof of Theorem \ref{theo_curvature_regular}. The values of probability measures $\mu_{e}^{\alpha}$ and $\mu_{a}^{\alpha}$ are
\begin{center}
\begin{minipage}{0.62\textwidth}
\[
\begin{alignedat}{2}
\mu_e^\alpha(e) & = \alpha,
&\qquad
\mu_a^\alpha(a) & = \alpha, \\
\mu_e^\alpha(a) & = \frac{1-\alpha}{2},
&\qquad
\mu_a^\alpha(a^2) & = \frac{1-\alpha}{2}, \\
\mu_e^\alpha(b) & = \frac{1-\alpha}{2},
&\qquad
\mu_a^\alpha(ba^{n-1}) & = \frac{1-\alpha}{2}.
\end{alignedat}
\]
\end{minipage}
\end{center}

    Next, note that 
\begin{enumerate}[i.]
    \item $d(a,a^2)=d(a,ba^{n-1})=1$.

    \item $d(b,a^2) =
     \begin{cases}
       2, &\quad\text{if } n=3 \text{ with shortest path } (b,ba,bab=a^2) \\
       3, &\quad\text{if } n\geq 4 \text{ with shortest path } (b,bb=e,a,ab=ba^{n-1}) \\
     \end{cases} $

      \item $d(b,ba^{n-1}) =
     \begin{cases}
       2, &\quad\text{if } n=3 \text{ with shortest path } (b,ba,ba^2) \\
       3, &\quad\text{if } n\geq 4 \text{ with shortest path } (b,bb=e,a,ab=ba^{n-1}) \\
     \end{cases} $
\end{enumerate}

For every $n \geq 3$, we can choose the pair $(a,a^2)$ and $(b,ba^{n-1})$ as

\begin{equation*}
    d(a,a^2)+d(b,ba^{n-1})= 
    \begin{cases}
       3, &\quad\text{if } n=3 \\
       4, &\quad\text{if } n\geq 4 \\
    \end{cases}
\end{equation*}
is minimum.
Then, the coupling $A$ is defined as $A(e,a)=\a, A(a,a^2) =A(b,ba^{n-1})=\frac{1-\a}{2}$ and $0$ otherwise. By simple computation, we have 
\begin{equation*}
     \sum_{v,w\in \cV} A(v,w)d(v,w)= 
    \begin{cases}
       1+\dfrac{1-\a}{2} , &\quad\text{if } n=3, \\
       2-\a , &\quad\text{if } n\geq 4. 
    \end{cases}
\end{equation*}  
From Definition \ref{def_Wasserstein}, it follows that 
\begin{enumerate}[(1).]
    \item For $n=3$, we have $\frac{1-W}{1-\a} \geq -\frac{1}{2}$, giving $\kappa(e,a) \geq -\frac{1}{2}$ by taking limit as $\alpha$ approach one.

    \item For $n\geq 4$, we have $\frac{1-W}{1-\a} \geq -1$. So $\kappa (e,a) \geq -1$. 
\end{enumerate}

Now, we illustrate how to define the function $f$ for $n\geq3$. For $n=3$,  define the function $f:\cV \rightarrow \R$ as follows: (See Figure \ref{fig:k(e,a)forD_n})
    \begin{enumerate}[(1).]
        \item Let $f(e)=1, f(a)=0$ and $f(w(y))=f(a^2)=-1$.
        
        \item By our definition, we have 
        \[
        f(b)= \min \{ f(e)+d(b,e), f(a)+d(b,a), f(a^2)+d(b,a^2) \}  
        = \min \{ 2,  2, 2 \} 
        =2
        \]
       Then, $f(ba^2)=f(b)-d(b,ba^2)=2-2=0$.
    \end{enumerate}
In summary, we have
    \begin{equation*}
        f(v) = 
     \begin{cases}
       1 &\quad\text{if } v=e \\
       0, &\quad\text{if } v=a \\
       -1, &\quad\text{if } v=a^2 \\
       2, &\quad\text{if } v=b  \\
       0, &\quad\text{if } v=ba^2. \\
     \end{cases}
    \end{equation*}
    It follows that 
    \[
    W \geq (1-0)\alpha + \frac{1-\alpha}{2} [(0-(-1))+(2-0)]  
    = 1+ \frac{1-\alpha}{2}.
    \]
    Hence, $\kappa(e,a) \leq -\frac{1}{2}$ and therefore $\kappa(e,a) = -\frac{1}{2}$. 

For $n\geq 4$, we define $f$ as follows: (See Figure \ref{fig:k(e,a)forD_n})
\begin{enumerate}[(1).]
        \item Let $f(e)=1, f(a)=0$ and $f(a^2)=-1$. 
        \item By our construction, 
        \[
        f(b)=\min \{f(e)+d(b,e), f(a)+d(b,a), f(a^2)+d(b,a^2)\} 
        = \min \{2, 2, 2  \} 
        =2. 
        \]        
        Then, $f(ba^{n-1})=f(b)-d(b,a^2)=2-3=-1$.
    \end{enumerate}

In short,
 \begin{equation*}
        f(v) = 
     \begin{cases}
       1 &\quad\text{if } v=e, \\
       0, &\quad\text{if } v=a, \\
       -1, &\quad\text{if } v=a^2, \\
       2, &\quad\text{if } v=b,  \\
       -1, &\quad\text{if } v=ba^{n-1}. \\
     \end{cases}
    \end{equation*}
    It follows that
    \[
    W\geq \alpha + \frac{1-\alpha}{2}(0-(-1)+2-(-1)) 
    =\alpha +2(1-\alpha) 
    = 2-\alpha.
    \]
    Hence, $\kappa (e,a) \leq -1$ and therefore $\kappa(e,a)=-1$.

    Let us verify our answer by using Theorem \ref{theo_curvature_regular}. For $n=3$, since $r=2$ and $\sum_{v\in N^{\text{out}}(x)} d(v,w(v))=3$, we have $\kappa(e,a)=1-\frac{1}{2}(3)=-\frac{1}{2}$. Meanwhile for $n\geq 4$, $\kappa(e,a)=1-\frac{1}{2}(4)=-1$.  

\begin{figure}[htbp]
    \centering
    \begin{subfigure}{0.48\textwidth}
        \centering
       \begin{tikzpicture}[
    scale=1.,
    vertex/.style={circle, fill=black, inner sep=0pt, minimum size=3pt},
    boxed/.style={draw, rectangle, inner sep=2pt, font=\small}
]

\def\outerR{2.5}
\def\innerR{1.2}

\draw[, thick, -{Stealth[bend]}] (90:\outerR) arc (90:-30:\outerR);
\draw[, thick, -{Stealth[bend]}] (330:\outerR) arc (330:210:\outerR);
\draw[, thick, -{Stealth[bend]}] (210:\outerR) arc (210:90:\outerR);

\draw[, thick, -{Stealth[bend]}] (90:\innerR) arc (90:210:\innerR);
\draw[, thick, -{Stealth[bend]}] (210:\innerR) arc (210:330:\innerR);
\draw[, thick, -{Stealth[bend]}] (330:\innerR) arc (330:450:\innerR);

\draw[, thick, Stealth-Stealth] (90:\outerR) -- (90:\innerR);
\draw[, thick, Stealth-Stealth] (330:\outerR) -- (330:\innerR);
\draw[, thick, Stealth-Stealth] (210:\outerR) -- (210:\innerR);

\node[vertex] at (90:\outerR) {};
\node[vertex] at (330:\outerR) {};
\node[vertex] at (210:\outerR) {};
\node[vertex] at (90:\innerR) {};
\node[vertex] at (330:\innerR) {};
\node[vertex] at (210:\innerR) {};

\node[, above=3pt] at (90:\outerR) {$e$};
\node[, right=3pt] at (330:\outerR) {$a$};
\node[, left=3pt] at (210:\outerR) {$a^2$};

\node[, above right=2pt] at (90:\innerR) {$b$};
\node[, above left=2pt] at (210:\innerR) {$ba$};
\node[, above right=3pt] at (330:\innerR) {$ba^2$};

\node[boxed] at (90:\outerR + 0.3) [xshift=0.4cm,yshift=0.1cm] {$1$};
\node[boxed] at (330:\outerR) [xshift=0.8cm] {$0$};
\node[boxed] at (210:\outerR) [xshift=-0.5cm, yshift=-0.5cm] {$-1$};
\node[boxed] at (310:\innerR) [xshift=1.3cm, yshift=0.8cm] {$0$};
\node[boxed] at (80:1.8) [xshift=0.4cm, yshift=-0.2cm] {$2$};

\end{tikzpicture}
        \caption{$\Gamma(D_3, \{a,b\}).$}
    \end{subfigure}
    \hfill
    \begin{subfigure}{0.48\textwidth}
        \centering
                \begin{tikzpicture}[
    scale=1.3,
    vertex/.style={circle, fill=black, inner sep=0pt, minimum size=3pt},
    boxed/.style={draw, rectangle, inner sep=2pt, font=\small}
]

\draw[, thick] (170:2.5) arc (170:-30:2.5);
\node[vertex] at (170:2.5) (v1o) {};
\node[vertex] at (130:2.5) (v2o) {};
\node[vertex] at (90:2.5)  (v3o) {};
\node[vertex] at (50:2.5)  (v4o) {};
\node[vertex] at (10:2.5)  (v5o) {};
\node[vertex] at (-30:2.5) (v6o) {};

\draw[, thick] (170:1.3) arc (170:-30:1.3);
\node[vertex] at (170:1.3) (v1i) {};
\node[vertex] at (130:1.3) (v2i) {};
\node[vertex] at (90:1.3)  (v3i) {};
\node[vertex] at (50:1.3)  (v4i) {};
\node[vertex] at (10:1.3)  (v5i) {};
\node[vertex] at (-30:1.3) (v6i) {};

\begin{scope}[, thick, -{Stealth[bend]}]
    \draw (170:2.5) arc (170:160:2.5);
    \draw (120:2.5) arc (120:105:2.5);
    \draw (90:2.5)  arc (90:65:2.5);
    \draw (50:2.5)  arc (50:25:2.5);
    \draw (10:2.5)  arc (10:-15:2.5);

    \draw (130:1.3) arc (130:155:1.3);
    \draw (90:1.3)  arc (90:115:1.3);
    \draw (50:1.3)  arc (50:75:1.3);
    \draw (10:1.3)  arc (10:35:1.3);
    \draw (-30:1.3)  arc (-30:-5:1.3);
\end{scope}

\draw[, thick, Stealth-Stealth] (130:2.5) -- (130:1.3); 
\draw[, thick, Stealth-Stealth] (90:2.5) -- (90:1.3);   
\draw[, thick, Stealth-Stealth] (50:2.5) -- (50:1.3);  
\draw[, thick, Stealth-Stealth] (10:2.5) -- (10:1.3);  

\node[, above left] at (v2o) {$a^{n-1}$};
\node[, above] at (v3o) {$e$};
\node[boxed] at (v3o) [xshift=0.4cm, yshift=0.4cm] {1};

\node[, right] at (v4o) {$a$};
\node[boxed] at (v4o) [xshift=0.7cm] {$0$};

\node[, right] at (v5o) {$a^2$};
\node[boxed] at (v5o) [xshift=1.0cm] {$-1$};

\node[, left= 0.3cm of v2i,] {$ba$};
\node[, above right] at (v3i) {$b$};
\node[boxed] at (v3i) [xshift=0.6cm, yshift=0.4cm] {2};
\node[, right = 0.15cm of v4i,]  {$ba^{n-1}$};
\node[boxed, ] at (v4i) [xshift=1.4cm, yshift=-0.4cm] {$-1$};

\node[, below right] at (v5i) {$ba^{n-2}$};


\node[, rotate=-20] at (175:2.5) [xshift=0.cm, yshift=0.0cm] {$\approx$};
\node[, rotate=-20] at (175:1.3) [xshift=0.03cm, yshift=-0.1cm] {$\approx$};
\node[, rotate=-20] at (-35:1.3) [xshift=0.03cm, yshift=-0.1cm] {$\approx$};
\node[, rotate=-20] at (-35:2.5) [xshift=0.03cm, yshift=0.0cm] {$\approx$};
\end{tikzpicture}
    \caption{$\Gamma(D_n, \{a,b\})$ for $n\geq 4.$}
    \end{subfigure}
    \caption{Values of $f:V \rightarrow \R$ for $\kappa(e,a).$}
    \label{fig:k(e,a)forD_n}
\end{figure}
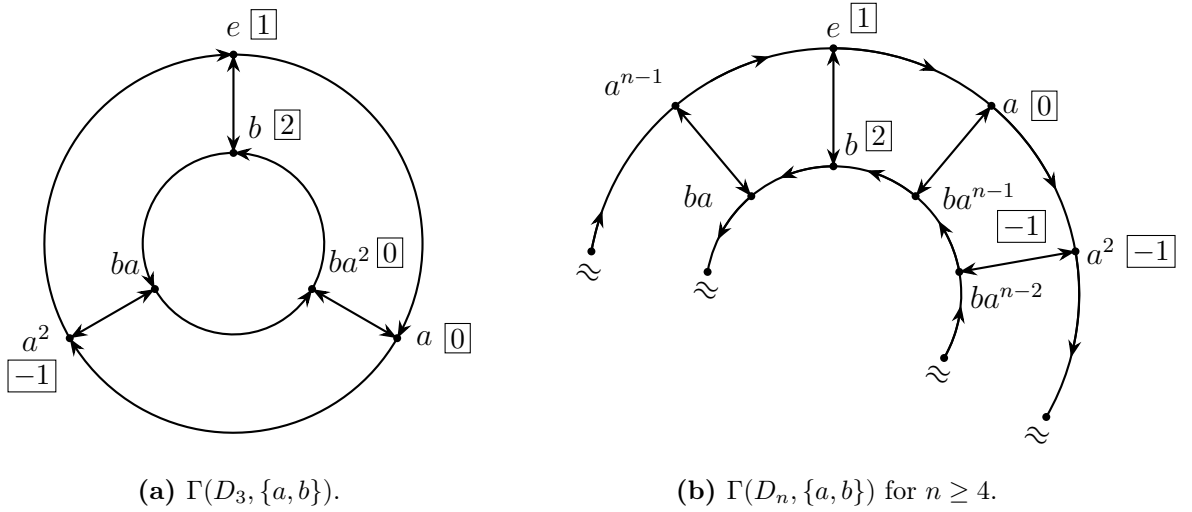


Next we compute $\kappa (e,b)$. For $n\geq 3$, we have $N^{\text{out}}(e)=\{a,b\}$ and $N^{\text{out}}(b)=\{e,ba\}$. Since $e\in N^{\text{out}}(a)$, we apply case (b) in the proof of Theorem \ref{theo_curvature_regular}. The values of probability measures $\mu_{e}^{\alpha}$ and $\mu_{b}^{\alpha}$ are
\begin{center}
\begin{minipage}{0.48\textwidth}
\[
\begin{alignedat}{2}
\mu_e^\alpha(e) & = \alpha,
&\qquad
\mu_a^\alpha(b) & = \alpha, \\
\mu_e^\alpha(b) & = \frac{1-\alpha}{2},
&\qquad
\mu_a^\alpha(e) & = \frac{1-\alpha}{2}, \\
\mu_e^\alpha(a) & = \frac{1-\alpha}{2},
&\qquad
\mu_a^\alpha(ba) & = \frac{1-\alpha}{2}.
\end{alignedat}
\]
\end{minipage}
\end{center}
Note that 
\begin{equation*}
d(a,ba) =
     \begin{cases}
       2, &\quad\text{if } n=3 \text{ with shortest path } (a,a^2,a^2b=ba) \\
       3, &\quad\text{if } n\geq 4 \text{ with shortest path } (a,ab=ba^{n-1},aba=b,ba) \\
     \end{cases} 
\end{equation*}

For every $n \geq 3$, we choose the pair $(a,ba)$ as they are the only pair available.
Then, the coupling $A$ is defined as $A(e,b)=\a-\frac{1-\a}{2}, A(e,e)=A(b,b)=A(a,ba)=\frac{1-\a}{2}$ and $0$ otherwise. By simple computation, we have 
\begin{equation*}
     \sum_{v,w\in \cV} A(v,w)d(v,w)= 
    \begin{cases}
       \dfrac{1+\a}{2} , &\quad\text{if } n=3, \\
       1 , &\quad\text{if } n\geq 4. 
    \end{cases}
\end{equation*}   

From Definition \ref{def_Wasserstein}, it follows that 
\begin{enumerate}
    \item For $n=3$, we have $\frac{1-W}{1-\a} \geq \frac{1}{2}$, giving $\kappa(e,b) \geq \frac{1}{2}$ by taking limit as $\alpha$ approach one.

    \item For $n\geq 4$, we have $\frac{1-W}{1-\a} \geq 0$. So $\kappa (e,b) \geq 0$. 
\end{enumerate}

For $n=3$, define the function $f:\cV \rightarrow \R$ as: (See Figure \ref{fig:k(e,b)forD_n})
 \begin{enumerate}
        \item Let $f(e)=1$ and $f(b)=0$.
        \item Then,
        $f(a)=\min \{f(e)+d(a,e), 
            f(b)+d(a,b)\} 
            =\min \{ 3, 2\} 
            =2. 
        $
        Hence, $f(ba)=f(a)-d(a,ba)=2-2=0$.
    \end{enumerate}
    
In summary, we have
    \begin{equation*}
        f(v) = 
     \begin{cases}
       1 &\quad\text{if } v=e, \\
       0, &\quad\text{if } v=b, \\
       2, &\quad\text{if } v=a,  \\
       0, &\quad\text{if } v=ba. \\
     \end{cases}
    \end{equation*}
    It follows that 
    \[
        W \geq (1-0)\alpha + \frac{1-\alpha}{2} [(0-1)+(2-0)] 
        = \frac{1+\alpha}{2}.
    \]
    Hence, $\kappa(e,a) \leq \frac{1}{2}$ and therefore $\kappa(e,a) = \frac{1}{2}$. 

For $n\geq 4$, we define $f$ as: (See Figure \ref{fig:k(e,b)forD_n})
   \begin{enumerate}
        \item Let $f(e)=1 $ and $ f(b)=0$. 
        \item Then,
        $f(a)=\min \{f(e)+d(a,e), f(b)+d(a,b)\} 
            = \min \{4, 2 \} 
            =2.$
        Hence, $f(ba)=f(a)-d(a,ba)=2-3=-1$.
    \end{enumerate}
In short,
 \begin{equation*}
        f(v) = 
     \begin{cases}
       1 &\quad\text{if } v=e, \\
       0, &\quad\text{if } v=b, \\
       2, &\quad\text{if } v=a,  \\
       -1, &\quad\text{if } v=ba. \\
     \end{cases}
    \end{equation*}

    It follows that
    \[
     W \geq  (1-0)\alpha+\frac{1-\alpha}{2} ((0-1)+ 2-(-1))  
     = \alpha+2\left(\frac{1-\alpha}{2} \right)
     =1.
    \]
    Hence, $\kappa(e,b) \leq 0$. Therefore, $\kappa(e,b)=0$. Alternatively, by Theorem \ref{theo_curvature_regular}, we have $\kappa(e,b)=1-\frac{1}{2}(2-1)=\frac{1}{2}$ for $n=3$ and $\kappa(e,b)=1-\frac{1}{2}(3-1)=0$ for $n \geq 4$.
\end{proof}

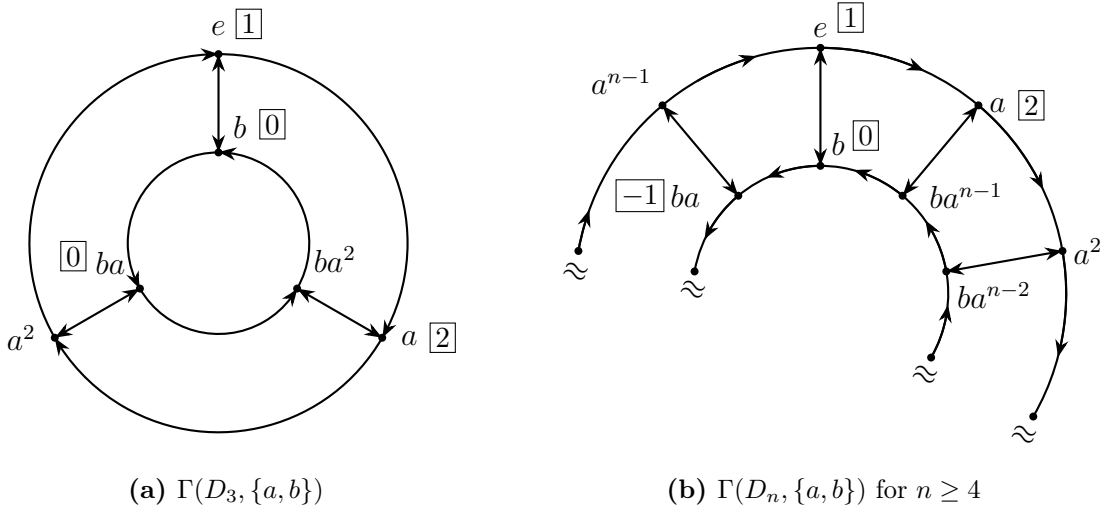
\begin{figure}[htbp]
    \centering
    \begin{subfigure}{0.48\textwidth}
        \centering
       \begin{tikzpicture}[
    scale=1.,
    vertex/.style={circle, fill=black, inner sep=0pt, minimum size=3pt},
    boxed/.style={draw, rectangle, inner sep=2pt, font=\small}
]

\def\outerR{2.5}
\def\innerR{1.2}

\draw[, thick, -{Stealth[bend]}] (90:\outerR) arc (90:-30:\outerR);
\draw[, thick, -{Stealth[bend]}] (330:\outerR) arc (330:210:\outerR);
\draw[, thick, -{Stealth[bend]}] (210:\outerR) arc (210:90:\outerR);

\draw[, thick, -{Stealth[bend]}] (90:\innerR) arc (90:210:\innerR);
\draw[, thick, -{Stealth[bend]}] (210:\innerR) arc (210:330:\innerR);
\draw[, thick, -{Stealth[bend]}] (330:\innerR) arc (330:450:\innerR);

\draw[, thick, Stealth-Stealth] (90:\outerR) -- (90:\innerR);
\draw[, thick, Stealth-Stealth] (330:\outerR) -- (330:\innerR);
\draw[, thick, Stealth-Stealth] (210:\outerR) -- (210:\innerR);

\node[vertex] at (90:\outerR) {};
\node[vertex] at (330:\outerR) {};
\node[vertex] at (210:\outerR) {};
\node[vertex] at (90:\innerR) {};
\node[vertex] at (330:\innerR) {};
\node[vertex] at (210:\innerR) {};

\node[, above=3pt] at (90:\outerR) {$e$};
\node[, right=3pt] at (330:\outerR) {$a$};
\node[, left=3pt] at (210:\outerR) {$a^2$};

\node[, above right=2pt] at (90:\innerR) {$b$};
\node[, above left=2pt] at (210:\innerR) {$ba$};
\node[, above right=3pt] at (330:\innerR) {$ba^2$};

\node[boxed] at (90:\outerR + 0.3) [xshift=0.4cm,yshift=0.1cm] {$1$};
\node[boxed] at (330:\outerR) [xshift=0.8cm] {$2$};
\node[boxed] at (220:\innerR) [xshift=-1.0cm, yshift=0.6cm] {$0$};
\node[boxed] at (80:1.8) [xshift=0.4cm, yshift=-0.2cm] {$0$};

\end{tikzpicture}
        \caption{$\Gamma(D_3, \{a,b\})$}
    \end{subfigure}
    \hfill
    \begin{subfigure}{0.48\textwidth}
        \centering
                \begin{tikzpicture}[
    scale=1.3,
    vertex/.style={circle, fill=black, inner sep=0pt, minimum size=3pt},
    boxed/.style={draw, rectangle, inner sep=2pt, font=\small}
]

\draw[, thick] (170:2.5) arc (170:-30:2.5);
\node[vertex] at (170:2.5) (v1o) {};
\node[vertex] at (130:2.5) (v2o) {};
\node[vertex] at (90:2.5)  (v3o) {};
\node[vertex] at (50:2.5)  (v4o) {};
\node[vertex] at (10:2.5)  (v5o) {};
\node[vertex] at (-30:2.5) (v6o) {};

\draw[, thick] (170:1.3) arc (170:-30:1.3);
\node[vertex] at (170:1.3) (v1i) {};
\node[vertex] at (130:1.3) (v2i) {};
\node[vertex] at (90:1.3)  (v3i) {};
\node[vertex] at (50:1.3)  (v4i) {};
\node[vertex] at (10:1.3)  (v5i) {};
\node[vertex] at (-30:1.3) (v6i) {};

\begin{scope}[, thick, -{Stealth[bend]}]
    \draw (170:2.5) arc (170:160:2.5);
    \draw (120:2.5) arc (120:105:2.5);
    \draw (90:2.5)  arc (90:65:2.5);
    \draw (50:2.5)  arc (50:25:2.5);
    \draw (10:2.5)  arc (10:-15:2.5);

    \draw (130:1.3) arc (130:155:1.3);
    \draw (90:1.3)  arc (90:115:1.3);
    \draw (50:1.3)  arc (50:75:1.3);
    \draw (10:1.3)  arc (10:35:1.3);
    \draw (-30:1.3)  arc (-30:-5:1.3);
\end{scope}

\draw[, thick, Stealth-Stealth] (130:2.5) -- (130:1.3); 
\draw[, thick, Stealth-Stealth] (90:2.5) -- (90:1.3);   
\draw[, thick, Stealth-Stealth] (50:2.5) -- (50:1.3);  
\draw[, thick, Stealth-Stealth] (10:2.5) -- (10:1.3);  

\node[, above left] at (v2o) {$a^{n-1}$};
\node[, above] at (v3o) {$e$};
\node[boxed] at (v3o) [xshift=0.4cm, yshift=0.4cm] {1};

\node[, right] at (v4o) {$a$};
\node[boxed] at (v4o) [xshift=0.7cm] {$2$};

\node[, right] at (v5o) {$a^2$};

\node[, left= 0.3cm of v2i,] {$ba$};
\node[boxed] at (v2i) [xshift=-1.3cm] {$-1$};

\node[, above right] at (v3i) {$b$};
\node[boxed] at (v3i) [xshift=0.6cm, yshift=0.4cm] {0};
\node[, right = 0.15cm of v4i,]  {$ba^{n-1}$};

\node[, below right] at (v5i) {$ba^{n-2}$};


\node[, rotate=-20] at (175:2.5) [xshift=0.cm, yshift=0.0cm] {$\approx$};
\node[, rotate=-20] at (175:1.3) [xshift=0.03cm, yshift=-0.1cm] {$\approx$};
\node[, rotate=-20] at (-35:1.3) [xshift=0.03cm, yshift=-0.1cm] {$\approx$};
\node[, rotate=-20] at (-35:2.5) [xshift=0.03cm, yshift=0.0cm] {$\approx$};

\end{tikzpicture}
    \caption{$\Gamma(D_n, \{a,b\})$ for $n\geq 4$}
    \end{subfigure}
    \caption{Values of $f:V \rightarrow \R$ for $\kappa(e,b).$}
    \label{fig:k(e,b)forD_n}
\end{figure}

\begin{proposition}
\label{prop_ricci_quartenion_{a,b}}
    The Lin-Lu-Yau Ricci curvature for the directed Cayley graph $\Gamma(Q_{4m}, \{a,b\})$ is
\begin{center}
\begin{tabular}{|l|c|c|c|c|c|l|} 
\hline
 & \multicolumn{5}{|c|}{$\Gamma(Q_{4m},\{a,b\})$}  \\
\hline
$\kappa(x,y)$ & $m=2$  &  $m=3$ & $m=4$  & $m=5$  & $m\geq 6$ \\ 
\hline
$\kappa(e,a)$    & $0$ & $-\frac{1}{2}$ & $-1$ & $-\frac{3}{2}$ & $-2$   \\
\hline
$\kappa(e,b)$  & $0$ & $-\frac{1}{2}$   & $-1$ & $-1$  & $-1$   \\
\hline
\end{tabular}
\end{center}
\end{proposition}
\begin{proof}
     The graph $\Gamma(Q_{4m}, S')$ is $2$-regular. We will compute $\kappa(e,a)$ first.  Note that $N^{\text{out}}(e)=\{a,b\}$ and  $N^{\text{out}}(a)=\{a^2,ab\}$.
    Then,
    \begin{enumerate}[(1).]
    \item $d(a,a^2)=d(a,ab)=1$.
    \item For $d(b,a^2)$, since $b(a^{m-2}b)=a^2$ or $b(b^3a^2)=a^2$, we have 
    \begin{align*}
    d(b,a^2) = \min \{m-1,5 \}   
    = \begin{cases}
       1, &\quad\text{if } m=2, \\
       2, &\quad\text{if } m=3,  \\
       3, &\quad\text{if } m=4, \\
       4, &\quad\text{if } m=5, \\
       5, &\quad\text{if } m\geq 6.  \\
     \end{cases} 
    \end{align*}
      \item For $d(b,ab)$, since $b(b^2a^{m-1}b)=ab$ or $b(b^3ab)=ab$, we have       
      \begin{align*}
      d(b,ab) = \min \{m+1, 5\} =
    \begin{cases}
       3, &\quad\text{if } m=2, \\
       4, &\quad\text{if } m=3,  \\
       5, &\quad\text{if } m\geq 4. \\
     \end{cases}   
      \end{align*}
\end{enumerate}

Since $d(b,a^2) \leq d(b,ab)$ for $m \geq 2$, we choose the pairs $(a,ab)$ and $(b,a^2)$. Then,
\begin{align*}
d(a,ab)+d(b,a^2) 
    = \begin{cases}
       2, &\quad\text{if } m=2, \\
       3, &\quad\text{if } m=3,  \\
       4, &\quad\text{if } m=4, \\
       5, &\quad\text{if } m=5, \\
       6, &\quad\text{if } m\geq 6.  \\
     \end{cases} 
\end{align*}

By Theorem \ref{theo_curvature_regular} (Case $e \notin N^{\text{out}}(a)$), we have  

\begin{align*}
\kappa(e,a)
    = \begin{cases}
       0, &\quad\text{if } m=2, \\
       -\frac{1}{2}, &\quad\text{if } m=3,  \\
       -1, &\quad\text{if } m=4, \\
       -\frac{3}{2}, &\quad\text{if }m=5, \\
       -2, &\quad\text{if } m\geq 6.  \\
     \end{cases} 
\end{align*}

To compute $\kappa(e,b)$, note $N^{\text{out}}(e)=\{a,b\}$ and $N^{\text{out}}(b)=\{ba=a^{2m-1}b, b^2=a^m\}$. 

 \begin{enumerate}[(1).]
    \item Since $a(ba^2)=ba$, it follows that $d(a,ba)=3$ for every $m\geq 2$. 

    \item From $a(bab)=b^2$ and $a(a^{m-1})=a^m=b^2$, we have
    \begin{align*}
   d(a, b^2)=\min \{m-1,3\}
     = \begin{cases}
       1, &\quad\text{if } m=2, \\
       2, &\quad\text{if } m=3,  \\
       3, &\quad\text{if } m\geq 4. \\
     \end{cases} 
     \end{align*}
    \item $d(b,ba)=d(b,b^2)=1$. 
\end{enumerate}

Thus, for $m\geq 4$, we choose the pairs $(a,b^2)$ and $(b,ba)$. This gives
\begin{align*}
d(a,b^2)+d(b,ba) 
    = \begin{cases}
       2, &\quad\text{if } m=2, \\
       3, &\quad\text{if } m=3,  \\
       4, &\quad\text{if } m\geq 4. \\
     \end{cases} 
\end{align*}

It follows that
\begin{align*}
\kappa(e,b)
    = \begin{cases}
       0, &\quad\text{if } m=2, \\
       -\frac{1}{2}, &\quad\text{if } m=3,  \\
       -1, &\quad\text{if } m\geq4. \\
     \end{cases} 
\end{align*}\qedhere
\end{proof}

\begin{proposition}
\label{prop_ricci_quartenion_3_a^{-1}}
    The Lin-lu-Yau Ricci curvature for the directed Cayley graphs $\Gamma(Q_{4m}, \{a,a^{-1},b\})$  is
\begin{center}
\begin{tabular}{|l|c|c|c|c|c|l|} 
\hline
 & \multicolumn{3}{|c|}{$\Gamma(Q_{4m},\{a,a^{-1},b\})$}\\
\hline
$\kappa(x,y)$ & $m=2$  &  $m=3$ & $m\geq4$  \\ 
\hline
$\kappa(e,a)$    & $\frac{2}{3}$ & $\frac{1}{3}$ & $-\frac{1}{3}$   \\
\hline
$\kappa(e,b)$  & $0$ & $0$   & $0$  \\
\hline
\end{tabular}
\end{center}
\end{proposition}

\begin{proof}
    If $S'=\{a,a^{-1}=a^{2m-1},b\}$, then $\Gamma(Q_{4m}, S')$ is $3$-regular. We will compute $\kappa(e,a)$ first. Note that $N^{\text{out}}(e)=\{a,a^{2m-1},b\}$ and  $N^{\text{out}}(a)=\{e,a^2,ab\}$.

Observe that 
 \begin{enumerate}[(1).]
    \item 
     \begin{align*}
    d(a^{2m-1},a^2)
    = \begin{cases}
       1, &\quad\text{if } m=2 \text{ as $a^3(a^{-1})=a^2$.} \\
       2, &\quad\text{if } m=3 \text{ as $a^5(b^2)=a^8=a^2$.} \\
       3, &\quad\text{if } m\geq 4 \text{ as $a^{2m-1}(a^3)=a^2$.}\\
     \end{cases} 
     \end{align*}

    \item From $a^{2m-1}(aab)=ab$, $d(a^{2m-1},ab)=3$.

    \item \begin{align*}
    d(b,a^2) = \min \{m-1,5 \}   
    = \begin{cases}
       1, &\quad\text{if } m=2 \\
       2, &\quad\text{if } m=3  \\
       3, &\quad\text{if } m=4 \\
       4, &\quad\text{if }m=5 \\
       5, &\quad\text{if } m\geq 6  \\
     \end{cases} 
    \end{align*}
 
    \item From $b(a^{-1})=ab$, we get $d(b,ab)=1$.  
\end{enumerate}
From above, we choose the pairs $(a^{2m-1},a^2)$ and $(b,ab)$. This gives
    \begin{align*}
    d(a^{2m-1},a^2)+d(b,ab) 
    = \begin{cases}
       2, &\quad\text{if } m=2 \\
       3, &\quad\text{if } m=3  \\
       4, &\quad\text{if } m\geq4 \\
     \end{cases} 
    \end{align*}

By Theorem \ref{theo_curvature_regular} (Case $e \in N^{\text{out}}(a)$)
\begin{align*}
\kappa(e,b)
    = \begin{cases}
       \frac{2}{3}, &\quad\text{if } m=2 \\
       \frac{1}{3}, &\quad\text{if } m=3  \\
       0, &\quad\text{if } m\geq4 \\
     \end{cases} 
\end{align*}

Now we compute $\kappa(e,b)$. Note that $N^{\text{out}}(e)=\{a,a^{-1}=a^{2m-1},b\}$ and  $N^{\text{out}}(b)=\{ba,ba^{-1},b^2\}$. Then,
 \begin{enumerate}[(1).]
    \item Since $a(ba^2)=ba$, it follows that $d(a,ba)=3$ for every $m\geq 2$. 

    \item Since $ba^{-1}=ab$, we have $d(a,b a^{-1})=1$ for every $m\geq 2$.

    \item From $a(bab)=b^2$ and $a(a^{m-1})=a^m=b^2$, we have
    \begin{align*}
   d(a, b^2)=\min \{m-1,3\}
     = \begin{cases}
       1, &\quad\text{if } m=2, \\
       2, &\quad\text{if } m=3,  \\
       3, &\quad\text{if } m\geq 4. \\
     \end{cases} 
     \end{align*}
     
    \item $d(b,ba)=d(ba^{-1})=d(b,b^2)=1$. 

    \item $d(a^{-1},ba)=1$, $d(a^{-1}, ba^{-1})=3$ and 
        \begin{align*}
   d(a^{-1}, b^2)=\min \{m-1,3\}
     = \begin{cases}
       1, &\quad\text{if } m=2, \\
       2, &\quad\text{if } m=3,  \\
       3, &\quad\text{if } m\geq 4. \\
     \end{cases} 
     \end{align*}
\end{enumerate}

For every $m \geq 2$, we consider the pairs $(a,ba^{-1}), (a^{-1},ba)$ and $(b,b^2)$ as $d(a,ba^{-1})+d(a^{-1},ba)+ d(b,b^2)=1+1+1=3$. Then, $\kappa(e,b)=1-\frac{1}{3}(3)=0$ for every $m\geq 2$.
\end{proof}

We consider now the generating set $S'$ being $\{a,b,b^{-1} \}$.
\begin{proposition}
\label{prop_ricci_quartenion_3_b^{-1}}
    The Ricci curvature for the directed Cayley graphs $\Gamma(Q_{4m}, \{a,b,b^{-1}\})$ 
\begin{center}
\begin{tabular}{|c|c|c|c|} 
\hline
 & \multicolumn{3}{|c|}{$\Gamma(Q_{4m},\{a,b,b^{-1}\}), m\geq 2 $}\\
\hline
$\kappa(x,y)$ & $m=2$  &  $m=3$ & $m\geq4$ \\
\hline
$\kappa(e,a)$  & $0$ & $-\frac{2}{3}$ & $-\frac{4}{3}$   \\
\hline
$\kappa(e,b)$  & $\frac{2}{3}$ & $\frac{1}{3}$ & $0$    \\
\hline
\end{tabular}
\end{center}
\end{proposition}

\begin{proof}
     Since $|\{a,b,b^{-1}\}|=3$, $\Gamma(Q_{4m}, S')$ is $3$-regular. We will compute $\kappa(e,a)$ first. Note that $N^{\text{out}}(e)=\{a,b,b^{-1}\}$ and  $N^{\text{out}}(a)=\{a^2,ab,ab^{-1}\}$. Then,

     \begin{enumerate}[(1).]
         \item $d(a,a^2)=d(a,ab)=d(a,ab^{-1})=1$.

         \item $d(b,a^2)=d(b,ab^{-1})=\min\{m-1,3\}$ and $d(b,ab)=3$.

         \item $d(b^{-1},a^2)=d(b^{-1},ab)=\min\{m-1,3\}$ and $d(b^{-1},ab^{-1})=3$.
     \end{enumerate}

     We choose the pairs $(a,a^2), (b,ab^{-1})$ and $(b^{-1}, ab)$. This gives 
     \begin{align*}
         \sum_{v \in N^{\text{out}}(e) } d(v,w(v))
         = \min \{ 2m-1,7\}
         =\begin{cases}
            3,  &\text{ if $m\geq 2$} \\
            5,  &\text{ if $m\geq 3$} \\
            7,  &\text{ if $m\geq 4$}
         \end{cases}
     \end{align*}
It follows that
 \begin{align*}
         \kappa (e,a)
         =\begin{cases}
            0,              & \text{ if $m= 2$} \\
            -\frac{2}{3},   & \text{ if $m= 3$} \\
            -\frac{4}{3},   & \text{ if $m\geq 4$}
         \end{cases}
     \end{align*}
By similar argument, one can compute $\kappa(e,b)$ as well. We omit the proof.
\end{proof}

The examples considered above concern directed Cayley graphs of dihedral groups and generalized quaternion groups with several non-trivial choices of generating sets. Since these groups, as well as many other groups, may admit further generating sets, we develop an algorithm for computing the Lin-Lu-Yau Ricci curvature of every arc $(e,s)$, where $s\in S$ and $S$ is a generating set of a group $G$. In particular, this algorithm applies not only to the families studied in this section, but also to common classes of groups such as $\mathbb{Z}_n$, symmetric groups $S_n$, and alternating groups $A_n$. The algorithm is presented in the Appendix below.

\section*{Acknowledgment}
Johnny Lim acknowledges the support from the Ministry of Higher Education Malaysia for Fundamental Research Grant Scheme with Project Code: 
\linebreak 
FRGS/1/2025/STG06/USM/02/1. Kevin Fung thanks Shi Kangli for the discussion and implementation of algorithm. \\


\noindent \textbf{Conflicts of interest.} \quad The authors declare no conflicts of interest. \\
\textbf{Data availability.} \quad Not applicable. \\

\bibliography{bibliography}{}
\bibliographystyle{amsplain}


\newpage
\appendix
\section{Table of the Lin-Lu-Yau Ricci curvature of directed Cayley graphs of certain group} \label{appendix_table}
Here are some remarks for the readers:
\begin{remark}
    \begin{enumerate}[1.]
        \item For generating set $S=\{s\}$ of size one, by Lemma \ref{lemma_curvature_single_generator}, we have $\kappa(e,s)=0$ or $2$. By our definition, $\kappa(e,t)=-\infty$ if $t$ cannot be generated by $s$. 

        \item Note that in the case of $\Gamma (\Z_n,\{1\})$. We have $\kappa(e,-1)=-\infty$. This is because $N^{\text{out}}(e)=\{1\}$ and $N^{\text{out}}(-1)=\{e\}$. As $1$ is of infinite order, $d(1,e)=\infty$. This example illustrates that it is possible to have $\kappa(e,s)=-\infty$ if it not possible to reach $w\in N^{\text{out}}(s)$ from any $v\in N^{\text{out}}(e)$. 

        \item For $G=D_n, Q_{4m}$ and $\Z_n$, the case of $S$ being the set of all generators of $G$ (so that $\Gamma(G,S)$ is undirected) and is symmetric are obtained in \cite{mizukaiRicci_Cayley2024}. We include those tables as well for the sake of completeness.

        \item The complete algorithm can be accessed from \url{https://github.com/Kevin12345-math/Algorithm-for-Lin-Lu-Yau-Curvature-of-Directed-Graphs/tree/main}.
    \end{enumerate}
\end{remark}



\begin{landscape} 
\begin{table}[p]  
\centering
\label{tab:graph_curvature}
\vspace{0.3cm}

\scalebox{1.0}{
\begin{tabular}{|c|c|c|c|c|}
\hline
\multicolumn{5}{|c|}{$\Gamma(Q_{4m}, -),\quad (|S|=1$)} \\
\hline
\multirow{2}{*}{$\kappa(x,y)$} & \multicolumn{1}{c|}{$\{a\}$} & \multicolumn{1}{c|}{$\{a^{-1}\}$} & \multicolumn{1}{c|}{$\{b\}$} & \multicolumn{1}{c|}{$\{b^{-1}\}$} \\
\cline{2-5}
 & $m\geq 2$ & $m\geq2$ & $m\geq2$ & $m\geq 2$ \\
\hline
$\kappa(e,a)$ & $0$ & $-\infty$ & $-\infty$ & $-\infty$ \\
\hline
$\kappa(e,a^{-1})$ & $-\infty$ & $0$ & $-\infty$ & $-\infty$ \\
\hline
$\kappa(e,b)$ & $-\infty$ & $-\infty$ & $0$ & $-\infty$ \\
\hline
$\kappa(e,b^{-1})$ & $-\infty$ & $-\infty$ & $-\infty$ & $0$ \\
\hline
\end{tabular}
}

\vspace{0.6cm} 

\scalebox{0.7}{
\begin{tabular}{|c|c|c|c|c|c|c|c|c|c|c|c|c|c|c|c|c|c|c|c|c|c|c|c|}
\hline
\multicolumn{24}{|c|}{$\Gamma(Q_{4m}, -),\quad (|S|=2)$} \\
\hline
\multirow{2}{*}{$\kappa(x,y)$} & \multicolumn{2}{c|}{$\{a,a^{-1}\}$} & \multicolumn{5}{c|}{$\{a,b\}$} & \multicolumn{5}{c|}{$\{a,b^{-1}\}$} & \multicolumn{5}{c|}{$\{a^{-1},b\}$} & \multicolumn{5}{c|}{$\{a^{-1},b^{-1}\}$} & \multicolumn{1}{c|}{$\{b,b^{-1}\}$} \\
\cline{2-24}
 & $m=2$ & $m\geq3$ & $m=2$ & $m=3$ & $m=4$ & $m=5$ & $m\geq 6$ & $m=2$ & $m=3$ & $m=4$ & $m=5$ & $m\geq 6$ & $m=2$ & $m=3$ & $m=4$ & $m=5$ & $m\geq 6$ & $m=2$ & $m=3$ & $m=4$ & $m=5$ & $m\geq 6$ & $m\geq2$ \\
\hline
$\kappa(e,a)$ & $1$ & $0$ & $0$ & $-\frac{1}{2}$ & $-1$ & $-\frac{3}{2}$ & $-2$ & $0$ & $-\frac{1}{2}$ & $-1$ & $-\frac{3}{2}$ & $-2$ & $-\infty$ & $-\infty$ & $-\infty$ & $-\infty$ & $-\infty$ & $-\infty$ & $-\infty$ & $-\infty$ & $-\infty$ & $-\infty$ & $-\infty$ \\
\hline
$\kappa(e,a^{-1})$ & $1$ & $0$ & $-\infty$ & $-\infty$ & $-\infty$ & $-\infty$ & $-\infty$ & $-\infty$ & $-\infty$ & $-\infty$ & $-\infty$ & $-\infty$ & $0$ & $-\frac{1}{2}$ & $-1$ & $-\frac{3}{2}$ & $-2$ & $0$ & $-\frac{1}{2}$ & $-1$ & $-\frac{3}{2}$ & $-2$ & $-\infty$ \\
\hline
$\kappa(e,b)$ & $-\infty$ & $-\infty$ & $0$ & $-\frac{1}{2}$ & $-1$ & $-1$ & $-1$ & $-\infty$ & $-\infty$ & $-\infty$ & $-\infty$ & $-\infty$ & $0$ & $-\frac{1;}{2}$ & $-1$ & $-1$ & $-1$ & $-\infty$ & $-\infty$ & $-\infty$ & $-\infty$ & $-\infty$ & $1$ \\
\hline
$\kappa(e,b^{-1})$ & $-\infty$ & $-\infty$ & $-\infty$ & $-\infty$ & $-\infty$ & $-\infty$ & $-\infty$ & $0$ & $-\frac{1}{2}$ & $-1$ & $-1$ & $-1$ & $-\infty$ & $-\infty$ & $-\infty$ & $-\infty$ & $-\infty$ & $0$ & $-\frac{1}{2}$ & $-1$ & $-1$ & $-1$ & $1$ \\
\hline
\end{tabular}%
}

\vspace{0.6cm}

\resizebox{\linewidth}{!}{%
\begin{tabular}{|l|c|c|c|c|c|c|c|c|c|c|c|c|}
\hline
\multicolumn{13}{|c|}{$\Gamma(Q_{4m}, -),\quad (|S|=3)$} \\
\hline
\multirow{2}{*}{$\kappa(x,y)$} & \multicolumn{3}{c|}{$\{a,a^{-1},b\}$} & \multicolumn{3}{c|}{$\{a,a^{-1},b^{-1}\}$} & \multicolumn{3}{c|}{$\{a,b,b^{-1}\}$} & \multicolumn{3}{c|}{$\{a^{-1},b,b^{-1}\}$} \\
\cline{2-13}
 & $m=2$ & $m=3$ & $m\geq 4$ & $m=2$ & $m=3$ & $m\geq 4$ & $m=2$ & $m=3$ & $m\geq 4$ & $m=2$ & $m=3$ & $m\geq 4$ \\
\hline
$\kappa(e,a)$ & $\frac{2}{3}$ & $\frac{1}{3}$ & $0$ & $\frac{2}{3}$ & $\frac{1}{3}$ & $0$ & $0$ & $-\frac{2}{3}$ & $-\frac{4}{3}$ & $-\infty$ & $-\infty$ & $-\infty$ \\
\hline
$\kappa(e,a^{-1})$ & $\frac{2}{3}$ & $\frac{1}{3}$ & $0$ & $\frac{2}{3}$ & $\frac{1}{3}$ & $0$ & $-\infty$ & $-\infty$ & $-\infty$ & $0$ & $-\frac{2}{3}$ & $-\frac{4}{3}$ \\
\hline
$\kappa(e,b)$ & $0$ & $0$ & $0$ & $-\infty$ & $-\infty$ & $-\infty$ & $\frac{2}{3}$ & $\frac{1}{3}$ & $0$ & $\frac{2}{3}$ & $\frac{1}{3}$ & $0$ \\
\hline
$\kappa(e,b^{-1})$ & $-\infty$ & $-\infty$ & $-\infty$ & $0$ & $0$ & $0$ & $\frac{2}{3}$ & $\frac{1}{3}$ & $0$ & $\frac{2}{3}$ & $\frac{1}{3}$ & $0$ \\
\hline
\end{tabular}%
}

\vspace{0.6cm}
\begin{center}
\begin{tabular}{|c|c|c|c|}
\hline
\multicolumn{4}{|c|}{$\Gamma(Q_{4m}, -),\quad (|S|=4)$} \\
\hline
\multirow{2}{*}{$\kappa(x,y)$} & \multicolumn{3}{c|}{$\{a,a^{-1},b,b^{-1}\}$} \\
\cline{2-4}
 & $m=2$ & $m=3$ & $m\geq 4$ \\
\hline
$\kappa(e,a)$ & $\frac{1}{2}$ & $\frac{1}{4}$ & $0$ \\
\hline
$\kappa(e,a^{-1})$ & $\frac{1}{2}$ & $\frac{1}{4}$ & $0$ \\
\hline
$\kappa(e,b)$ & $\frac{1}{2}$ & $\frac{1}{2}$ & $\frac{1}{2}$ \\
\hline
$\kappa(e,b^{-1})$ & $\frac{1}{2}$ & $\frac{1}{2}$ & $\frac{1}{2}$ \\
\hline
\end{tabular}
\end{center}
\end{table}
\end{landscape}



\begin{center}
\begin{tabular}{|c|c|c|c|}
\hline
\multicolumn{4}{|c|}{$\Gamma(D_n, -)\quad (|S|=1: \text{singletons})$} \\
\hline
\multirow{2}{*}{$\kappa(x,y)$} & \multicolumn{1}{c|}{$\{a\}$} & \multicolumn{1}{c|}{$\{a^{-1}\}$} & \multicolumn{1}{c|}{$\{b\}$} \\
\cline{2-4}
 & $n\geq 3$ & $n\geq 3$ & $n\geq 3$ \\
\hline
$\kappa(e,a)$ & $0$ & $-\infty$ & $-\infty$ \\
\hline
$\kappa(e,a^{-1})$ & $-\infty$ & $0$ & $-\infty$ \\
\hline
$\kappa(e,b)$ & $-\infty$ & $-\infty$ & $2$\\
\hline
\end{tabular}
\end{center}
\vspace{0.6cm}
\begin{center}
\begin{tabular}{|c|c|c|c|c|c|c|c|c|}
\hline
\multicolumn{9}{|c|}{$\Gamma(D_n, -)\quad (|S|=2: \text{pairs})$} \\
\hline
\multirow{2}{*}{$\kappa(x,y)$} & \multicolumn{4}{c|}{$\{a,a^{-1}\}$} & \multicolumn{2}{c|}{$\{a,b\}$} & \multicolumn{2}{c|}{$\{a^{-1},b\}$} \\
\cline{2-9}
 & $n=3$ & $n=4$ & $n=5$ & $n\geq 6$ & $n=3$ & $n\geq 4$ & $n=3$ & $n\geq 4$ \\
\hline
$\kappa(e,a)$ & $\frac{3}{2}$ & $1$ & $\frac{1}{2}$ & $0$ & $-\frac{1}{2}$ & $-1$ & $-\infty$ & $-\infty$ \\
\hline
$\kappa(e,a^{-1})$ & $\frac{3}{2}$ & $1$ & $\frac{1}{2}$ & $0$ & $-\infty$ & $-\infty$ & $-\frac{1}{2}$ & $-1$ \\
\hline
$\kappa(e,b)$ & $-\infty$ & $-\infty$ & $-\infty$ & $-\infty$ & $\frac{1}{2}$ & $0$ & $\frac{1}{2}$ & $0$ \\
\hline
\end{tabular}
\end{center}
\vspace{0.6cm}
\begin{center}
\begin{tabular}{|c|c|c|c|c|}
\hline
\multicolumn{5}{|c|}{$\Gamma(D_n, -)\quad (|S|=3: \text{triples})$} \\
\hline
\multirow{2}{*}{$\kappa(x,y)$} & \multicolumn{4}{c|}{$\{a,a^{-1},b\}$} \\
\cline{2-5}
 & $n=3$ & $n=4$ & $n=5$ & $n\geq 6$ \\
\hline
$\kappa(e,a)$ & $1$ & $\frac{2}{3}$ & $\frac{1}{3}$ & $0$ \\
\hline
$\kappa(e,a^{-1})$ & $1$ & $\frac{2}{3}$ & $\frac{1}{3}$ & $0$ \\
\hline
$\kappa(e,b)$ & $\frac{2}{3}$ & $\frac{2}{3}$ & $\frac{2}{3}$ & $\frac{2}{3}$ \\
\hline
\end{tabular}
\end{center}


\newpage

\begin{landscape}
\begin{table}[p] 
\centering
\label{tab:zn_curvature}
\vspace{0.2cm}

\resizebox{\linewidth}{!}{
\begin{tabular}{|c|c|c|c|c|c|c|c|c|c|c|c|c|c|c|}
\hline
\multicolumn{7}{|c|}{$\Gamma(\mathbb{Z}_n (k=2), -),\quad (|S|=1)$} \\
\hline
\multirow{2}{*}{$\kappa(x,y)$} & \multicolumn{1}{c|}{$\{+1\}$} & \multicolumn{1}{c|}{$\{-1\}$} & \multicolumn{2}{c|}{$\{+k\}$} & \multicolumn{2}{c|}{$\{-k\}$} \\
\cline{2-7}
 & $n\geq6$ & $n\geq6$ & $n\geq2,n \text{ even}$  & $n\geq3, n=2p+1 \text{, odd}$ & $n\geq 2, n \text{ even}$ & $n\geq 3, n=2p+1, n \text{, odd}$ \\
\hline
$\kappa(e,+1)$ & $0$ & $-\infty$ & $-\infty$ & $-p$ & $-\infty$ & $1-p$ \\
\hline
$\kappa(e,-1)$ & $-\infty$ & $0$ & $-\infty$ & $1-p$ & $-\infty$ & $-p$\\
\hline
$\kappa(e,+k)$ & $-1$ & $3-n$ & $0$ & $0$ & $-\infty$ & $-\infty$ \\
\hline
$\kappa(e,-k)$ & $3-n$ & $-1$ & $-\infty$ & $-\infty$ & $0$ & $0$ \\
\hline
\end{tabular}%
}
\vspace{0.6cm}

\resizebox{\linewidth}{!}{%
\begin{tabular}{|c|c|c|c|c|c|c|c|c|c|c|c|c|c|c|c|c|c|c|c|c|c|c|c|c|c|}
\hline
\multicolumn{18}{|c|}{$\Gamma(\mathbb{Z}_n (k=2), -),\quad (|S|=2)$} \\
\hline
\multirow{2}{*}{$\kappa(x,y)$} & \multicolumn{3}{c|}{$\{+1,-1\}$} & \multicolumn{1}{c|}{$\{+1,+k\}$} & \multicolumn{1}{c|}{$\{+1,-k\}$} & \multicolumn{1}{c|}{$\{-1,+k\}$} & \multicolumn{1}{c|}{$\{-1,-k\}$} & \multicolumn{10}{c|}{$\{+k,-k\}$} \\
\cline{2-18}
 & $n=6$ & $n=7$ & $n\geq 8$ & $n\geq6$ & $n\geq6$ & $n\geq 6$ & $n\geq6$ & $n=6$ & $n=7$ & $n=8$ & $n=9$ & $n=10$ & $n=11$ & $n=12$ & $n=13$ & $n=14$ & $n=15$ \\
\hline
$\kappa(e,+1)$ & $0$ & $0$ & $0$ & $\frac{1}{2}$ & $0$  & $-\infty$ & $-\infty$ & $-\infty$ & $-\frac{1}{2}$ & $-\infty$ & $-\frac{3}{2}$ & $-\infty$ & $-\frac{5}{2}$ & $-\infty$ & $-\frac{7}{2}$ & $-\infty$ & $-\frac{9}{2}$ \\
\hline
$\kappa(e,-1)$ & $0$ & $0$ & $0$ & $-\infty$ & $-\infty$ & $0$ & $\frac{1}{2}$ & $-\infty$ & $-\frac{1}{2}$ & $-\infty$ & $-\frac{3}{2}$ & $-\infty$ & $-\frac{5}{2}$ & $-\infty$ & $-\frac{7}{2}$ & $-\infty$ & $-\frac{9}{2}$ \\
\hline
$\kappa(e,+k)$ & $0$ & $-\frac{1}{2}$ & $-1$ & $0$ & $-\infty$ & $0$ & $-\infty$ & $\frac{3}{2}$ & $0$ & $1$ & $0$ & $\frac{1}{2}$ & $0$ & $0$ & $0$ & $0$ & $0$ \\
\hline
$\kappa(e,-k)$ & $0$ & $-\frac{1}{2}$ & $-1$ & $-\infty$ & $0$ & $-\infty$ & $0$ & $\frac{3}{2}$ & $0$ & $1$ & $0$ & $\frac{1}{2}$ & $0$ & $0$ & $0$ & $0$ & $0$ \\
\hline
\end{tabular}%
}

\vspace{0.6cm}

\resizebox{\linewidth}{!}{%
\begin{tabular}{|c|c|c|c|c|c|c|c|c|c|c|c|c|c|c|c|c|c|c|c|c|c|c|c|c|}
\hline
\multicolumn{25}{|c|}{$\Gamma(\mathbb{Z}_n (k=2), -),\quad (|S|=3)$} \\
\hline
\multirow{2}{*}{$\kappa(x,y)$} & \multicolumn{6}{c|}{$\{+1,-1,+k\}$} & \multicolumn{6}{c|}{$\{+1,-1,-k\}$} & \multicolumn{6}{c|}{$\{+1,+k,-k\}$} & \multicolumn{6}{c|}{$\{-1,+k,-k\}$} \\
\cline{2-25}
 & $n=6$ & $n=7$ & $n=8$ & $n=9$ & $n=10$ & $n\geq 11$ & $n=6$ & $n=7$ & $n=8$ & $n=9$ & $n=10$ & $n\geq 11$ & $n=6$ & $n=7$ & $n=8$ & $n=9$ & $n=10$ & $n\geq 11$ & $n=6$ & $n=7$ & $n=8$ & $n=9$ & $n=10$ & $n\geq 11$ \\
\hline
$\kappa(e,+1)$ & $\frac{2}{3}$ & $\frac{2}{3}$ & $\frac{2}{3}$ & $\frac{2}{3}$ & $\frac{2}{3}$ & $\frac{2}{3}$ & $1$ & $\frac{2}{3}$ & $\frac{2}{3}$ & $\frac{1}{3}$ & $\frac{1}{3}$ & $0$ & $\frac{1}{3}$ & $\frac{1}{3}$ & $\frac{1}{3}$ & $\frac{1}{3}$ & $\frac{1}{3}$ & $\frac{1}{3}$ & $-\infty$ & $-\infty$ & $-\infty$ & $-\infty$ & $-\infty$ & $-\infty$ \\
\hline
$\kappa(e,-1)$ & $1$ & $\frac{2}{3}$ & $\frac{2}{3}$ & $\frac{1}{3}$ & $\frac{1}{3}$ & $0$ & $\frac{2}{3}$ & $\frac{2}{3}$ & $\frac{2}{3}$ & $\frac{2}{3}$ & $\frac{2}{3}$ & $\frac{2}{3}$ & $-\infty$ & $-\infty$ & $-\infty$ & $-\infty$ & $-\infty$ & $-\infty$ & $\frac{1}{3}$ & $\frac{1}{3}$ & $\frac{1}{3}$ & $\frac{1}{3}$ & $\frac{1}{3}$ & $\frac{1}{3}$ \\
\hline
$\kappa(e,+k)$ & $\frac{1}{3}$ & $0$ & $0$ & $0$ & $0$ & $0$ & $-\infty$ & $-\infty$ & $-\infty$ & $-\infty$ & $-\infty$ & $-\infty$ & $1$ & $\frac{1}{3}$ & $\frac{2}{3}$ & $0$ & $\frac{1}{3}$ & $0$ & $1$ & $\frac{2}{3}$ & $\frac{2}{3}$ & $\frac{1}{3}$ & $\frac{1}{3}$ & $0$ \\
\hline
$\kappa(e,-k)$ & $-\infty$ & $-\infty$ & $-\infty$ & $-\infty$ & $-\infty$ & $-\infty$ & $\frac{1}{3}$ & $0$ & $0$ & $0$ & $0$ & $0$ & $1$ & $\frac{2}{3}$ & $\frac{2}{3}$ & $\frac{1}{3}$ & $\frac{1}{3}$ & $0$ & $1$ & $\frac{1}{3}$ & $\frac{2}{3}$ & $0$ & $\frac{1}{3}$ & $0$ \\
\hline
\end{tabular}%
}

\vspace{0.6cm}

\begin{center}
\resizebox{0.5\linewidth}{!}{%
\begin{tabular}{|c|c|c|c|c|c|c|}
\hline
\multicolumn{7}{|c|}{$\Gamma(\mathbb{Z}_n (k=2), -),\quad (|S|=4)$} \\
\hline
\multirow{2}{*}{$\kappa(x,y)$} & \multicolumn{6}{c|}{$\{+1,-1,+k,-k\}$} \\
\cline{2-7}
 & $n=6$ & $n=7$ & $n=8$ & $n=9$ & $n=10$ & $n\geq 11$ \\
\hline
$\kappa(e,+1)$ & $1$ & $1$ & $\frac{3}{4}$ & $\frac{3}{4}$ & $\frac{1}{2}$ & $\frac{1}{2}$ \\
\hline
$\kappa(e,-1)$ & $1$ & $1$ & $\frac{3}{4}$ & $\frac{3}{4}$ & $\frac{1}{2}$ & $\frac{1}{2}$ \\
\hline
$\kappa(e,+k)$ & $1$ & $\frac{3}{4}$ & $\frac{1}{2}$ & $\frac{1}{4}$ & $\frac{1}{4}$ & $0$ \\
\hline
$\kappa(e,-k)$ & $1$ & $\frac{3}{4}$ & $\frac{1}{2}$ & $\frac{1}{4}$ & $\frac{1}{4}$ & $0$ \\
\hline
\end{tabular}%
}
\end{center}

\end{table}
\end{landscape}
\newpage
\begin{center}
\begin{minipage}[t]{0.55\textwidth}
\centering
\begin{tabular}{|c|c|c|}
\hline
\multicolumn{3}{|c|}{$\Gamma(S_n, -),\quad (|S|=1)$} \\
\hline
\multirow{2}{*}{$\kappa(x,y)$} 
& \multicolumn{1}{c|}{$\{\sigma_{1}\}$} 
& \multicolumn{1}{c|}{$\{\sigma_{2}\}$} \\
\cline{2-3}
& $n\geq 3$ & $n\geq 3$ \\
\hline
$\kappa(e,\sigma_{1})$ & $2$ & $-\infty$ \\
\hline
$\kappa(e,\sigma_{2})$ & $-\infty$ & $2$ \\
\hline
\end{tabular}
\end{minipage}
\begin{minipage}[t]{0.38\textwidth}
\centering
\begin{tabular}{|c|c|}
\hline
\multicolumn{2}{|c|}{$\Gamma(S_n, -),\quad (|S|=2)$} \\
\hline
\multirow{2}{*}{$\kappa(x,y)$} 
& \multicolumn{1}{c|}{$\{\sigma_{1},\sigma_{2}\}$} \\
\cline{2-2}
& $n\geq 3$ \\
\hline
$\kappa(e,\sigma_{1})$ & $0$ \\
\hline
$\kappa(e,\sigma_{2})$ & $0$ \\
\hline
\end{tabular}
\end{minipage}
\end{center}

\vspace{0.6cm}


\begin{center}
\begin{tabular}{|c|c|c|c|c|}
\hline
\multicolumn{5}{|c|}{$\Gamma(A_n, -),\quad (|S|=1)$} \\
\hline
\multirow{2}{*}{$\kappa(x,y)$} & \multicolumn{1}{c|}{$\{V_{1}\}$} & \multicolumn{1}{c|}{$\{V_{1}^{-1}\}$} & \multicolumn{1}{c|}{$\{V_{2}\}$} & \multicolumn{1}{c|}{$\{V_{2}^{-1}\}$} \\
\cline{2-5}
 & $n\geq 4$ & $n\geq 4$ & $n\geq 4$ & $n\geq 4$ \\
\hline
$\kappa(e,V_{1})$ & $0$ & $-\infty$ & $-\infty$ & $-\infty$ \\
\hline
$\kappa(e,V_{1}^{-1})$ & $-\infty$ & $0$ & $-\infty$ & $-\infty$ \\
\hline
$\kappa(e,V_{2})$ & $-\infty$ & $-\infty$ & $0$ & $-\infty$ \\
\hline
$\kappa(e,V_{2}^{-1})$ & $-\infty$ & $-\infty$ & $-\infty$ & $0$ \\
\hline
\end{tabular}
\end{center}
\vspace{0.6cm}
\begin{center}
\begin{tabular}{|c|c|c|c|c|c|c|}
\hline
\multicolumn{7}{|c|}{$\Gamma(A_n, -),\quad (|S|=2)$} \\
\hline
\multirow{2}{*}{$\kappa(x,y)$} & \multicolumn{1}{c|}{$\{V_{1},V_{1}^{-1}\}$} & \multicolumn{1}{c|}{$\{V_{1},V_{2}\}$} & \multicolumn{1}{c|}{$\{V_{1},V_{2}^{-1}\}$} & \multicolumn{1}{c|}{$\{V_{1}^{-1},V_{2}\}$} & \multicolumn{1}{c|}{$\{V_{1}^{-1},V_{2}^{-1}\}$} & \multicolumn{1}{c|}{$\{V_{2},V_{2}^{-1}\}$} \\
\cline{2-7}
 & $n\geq 4$ & $n\geq 4$ & $n\geq 4$ & $n\geq 4$ & $n\geq 4$ & $n\geq 4$ \\
\hline
$\kappa(e,V_{1})$ & $\frac{3}{2}$ & $-\frac{1}{2}$ & $-\frac{3}{2}$ & $-\infty$ & $-\infty$ & $-\infty$ \\
\hline
$\kappa(e,V_{1}^{-1})$ & $\frac{3}{2}$ & $-\infty$ & $-\infty$ & $-\frac{3}{2}$ & $-\frac{1}{2}$ & $-\infty$ \\
\hline
$\kappa(e,V_{2})$ & $-\infty$ & $-\frac{1}{2}$ & $-\infty$ & $-\frac{3}{2}$ & $-\infty$ & $\frac{3}{2}$ \\
\hline
$\kappa(e,V_{2}^{-1})$ & $-\infty$ & $-\infty$ & $-\frac{3}{2}$ & $-\infty$ & $-\frac{1}{2}$ & $\frac{3}{2}$ \\
\hline
\end{tabular}
\end{center}
\vspace{0.6cm}
\begin{center}
\begin{tabular}{|c|c|c|c|c|}
\hline
\multicolumn{5}{|c|}{$\Gamma(A_n, -),\quad (|S|=3)$} \\
\hline
\multirow{2}{*}{$\kappa(x,y)$} & \multicolumn{1}{c|}{$\{V_{1},V_{1}^{-1},V_{2}\}$} & \multicolumn{1}{c|}{$\{V_{1},V_{1}^{-1},V_{2}^{-1}\}$} & \multicolumn{1}{c|}{$\{V_{1},V_{2},V_{2}^{-1}\}$} & \multicolumn{1}{c|}{$\{V_{1}^{-1},V_{2},V_{2}^{-1}\}$} \\
\cline{2-5}
 & $n\geq 4$ & $n\geq 4$ & $n\geq 4$ & $n\geq 4$ \\
\hline
$\kappa(e,V_{1})$ & $\frac{2}{3}$ & $\frac{1}{3}$ & $-\frac{2}{3}$ & $-\infty$ \\
\hline
$\kappa(e,V_{1}^{-1})$ & $\frac{1}{3}$ & $\frac{2}{3}$ & $-\infty$ & $-\frac{2}{3}$ \\
\hline
$\kappa(e,V_{2})$ & $-\frac{2}{3}$ & $-\infty$ & $\frac{2}{3}$ & $\frac{1}{3}$ \\
\hline
$\kappa(e,V_{2}^{-1})$ & $-\infty$ & $-\frac{2}{3}$ & $\frac{1}{3}$ & $\frac{2}{3}$ \\
\hline
\end{tabular}
\end{center}
\vspace{0.6cm}
\begin{center}
\begin{tabular}{|c|c|}
\hline
\multicolumn{2}{|c|}{$\Gamma(A_n, -),\quad (|S|=4)$} \\
\hline
\multirow{2}{*}{$\kappa(x,y)$} & \multicolumn{1}{c|}{$\{V_{1},V_{1}^{-1},V_{2},V_{2}^{-1}\}$} \\
\cline{2-2}
 & $n\geq 4$ \\
\hline
$\kappa(e,V_{1})$ & $\frac{1}{4}$ \\
\hline
$\kappa(e,V_{1}^{-1})$ & $\frac{1}{4}$ \\
\hline
$\kappa(e,V_{2})$ & $\frac{1}{4}$ \\
\hline
$\kappa(e,V_{2}^{-1})$ & $\frac{1}{4}$ \\
\hline
\end{tabular}
\end{center}

\vspace{1cm}




\end{document}